\documentclass[12pt]{amsart}

\usepackage{amssymb,amsmath,bm}
\usepackage{amsthm}
\usepackage{graphicx}        

\renewcommand{\setminus}{{\smallsetminus}}

\newcommand{\cp}[1]{\vcenter{\hbox{#1}}}

\usepackage{hyperref}

\newtheorem{theorem}{Theorem}[section]
\newtheorem{lemma}[theorem]{Lemma}
\newtheorem{proposition}[theorem]{Proposition}
\newtheorem{definition}[theorem]{Definition}
\newtheorem{corollary}[theorem]{Corollary}
\newtheorem{conjecture}[theorem]{Conjecture}

\theoremstyle{remark}
\newtheorem{remark}[theorem]{Remark}

\numberwithin{equation}{section}

\newcommand{\PSL}{\mathrm{PSL}_2}
\newcommand{\SL}{\mathrm{SL}_2}

\newcommand{\f}{\frac}

\newcommand\Vol{\operatorname{vol}}
\newcommand \CC{\mathbb C}

\renewcommand\SS{{\mathbb S}}
\newcommand\RR{{\mathbb R}}
\newcommand\ZZ{{\mathbb Z}}

\newcommand \id{\operatorname{Id}}

\newcommand{\li}{\operatorname{li_2}}

\newcommand \Hess{\operatorname{Hess}}

\newcommand\tr{\operatorname{Trace}}

\title[The BWY volume conjecture for the 4-puncture sphere]{The Bonahon-Wong-Yang volume conjecture \\ for the four-puncture sphere}
\author{Tushar Pandey}
\date{}

\begin{document}

\maketitle

\begin{abstract}
We compute the Bonahon-Wong-Yang quantum invariant for self-diffeomorphisms of the four-puncture sphere explicitly, based on the representation theory of the Checkov-Fock algebra.
As an application of the computation, we verify the volume conjecture proposed by Bonahon-Wong-Yang for the four puncture sphere bundles with some technical conditions. 
Together with the one-puncture torus bundles \cite{BWY3}, this is the first family of examples for which the conjecture is verified providing evidence for the conjecture.
\end{abstract}

\tableofcontents

\section{Introduction}
The original Volume Conjecture of Kashaev-Murakami-Murakami \cite{K97, MM01} proposes a relationship between the evaluation of N-th colored Jones polynomial of a link in $\SS^3$ at a root of unity $q = \exp\left(\frac{2\pi i}{N}\right)$ with the hyperbolic volume of the complement of the link. The current paper is devoted to another volume conjecture, developed by Bonahon-Wong-Yang \cite{BWY1} as a toy model for the original Volume Conjecture.
This conjecture is based on recent developments of the representation theory of the Kauffman bracket skein algebra $\mathcal{K}^q(S)$ of a surface \cite{BW16, FKBL19, GJS19}, a mathematical object arising from the theory underlying the Jones polynomial of a link in $\SS^3$.

Let $S$ be a surface with negative Euler characteristic $\chi(S)<0$, and consider an orientation-preserving pseudo-Anosov diffeomorphism $\varphi\colon S \rightarrow S$.
Let $[r] \in \mathcal{X}_{\SL(\CC)}(S)$ be an irreducible $ \varphi-$invariant character, namely represented by an irreducible group homomorphism $r \colon \pi_1(S) \rightarrow \SL(\CC)$ such that $[r \circ \varphi^*] = [r]$, where $\varphi^* \colon  \pi_1(S) \rightarrow \pi_1(S)$ is induced by the diffeomorphism $\varphi$. 
The combination of Thurston's Hyperbolization Theorem and Thurston's Hyperbolic Dehn Filling Theorem \cite{BP} , applied to the mapping torus $M_{\varphi} = S \times [0,1]/(x,1) \sim (\varphi(x),0)$, shows that there exist many such invariant characters $[r]$. 

For $n$ odd select, at each puncture $v$ of $S$, a weight $p_v \in \CC$ such that $T_n(p_v) = -\tr r(\alpha_v) $ where $\alpha_v$ is a loop going around the puncture $v$.
Then, for the quantum parameter $q = \exp \left(\frac{2\pi i}{n} \right)$, results from \cite{BW16,FKBL19,GJS19} associate to this data an irreducible representation $\rho_r \colon \mathcal{K}^q(S) \rightarrow \mathrm{End}(V)$ of the Kauffman bracket skein algebra $\mathcal{K}^q(S)$. 
If, in addition, we choose the puncture weights $p_v$ to be $\varphi$-invariant, namely $p_v = p_{\varphi(v)}$, the uniqueness part of these results shows that $\rho_r \colon \mathcal{K}^q(S) \rightarrow \mathrm{End}(V)$ and $\rho_r \circ \varphi_{*}\colon \mathcal{K}^q(S) \rightarrow \mathrm{End}(V)$ are isomorphic, where $\varphi_*\colon \mathcal{K}^q(S) \rightarrow \mathcal{K}^q(S)$ is defined by $\varphi_* ([K]) = [(\varphi \times Id_{[0,1]} )(K)]$ for every $[K] \in \mathcal{K}^q(S)$ represented by a framed link $K \subset S \times [0,1]$. For punctured surfaces, this $r$ is generic, that is, there exists a Zariski dense subset where the uniqueness holds, while for closed surfaces, the result holds for all irreducible $r$.
Namely, there exists a linear isomorphism $\Lambda_{\varphi, r}^q \colon V \rightarrow V$ such that 
\begin{align*}
    (\rho_r \circ \varphi_*)(X) = \Lambda_{\varphi,r}^q \circ \rho_r(X) \circ (\Lambda_{\varphi,r}^q)^{-1}
\end{align*}
for every $X \in \mathcal{K}^q(S) $.
By irreducibility of $\rho_r$, this intertwiner $\Lambda_{\varphi,r}^q$ is uniquely determined up to conjugation and scalar multiplication. In particular, if we normalize it so that $\det \Lambda_{\varphi,r}^q = 1$, the modulus $|\tr \Lambda_{\varphi,r}^q|$ is uniquely determined.

For each puncture $v$ with a loop $\alpha_v$ going around it, write the eigenvalues of $r(\alpha_v) \in \SL(\CC) $ as $-e^{\pm \theta_v}$ for some $\theta_v \in \CC$.
In particular, $\tr r(\alpha_v) = -e^{\theta_v} - e^{-\theta_v} $.
Then, an elementary property of the Chebychev polynomial $T_n$ is that the solutions of the equation 
\begin{align*}
    T_n(p_v) = -\tr r(\alpha_v)
\end{align*}
are all numbers of the form 
$p_v = t + t^{-1} $ with $t^n = e^{\theta_v}$. 
In particular, we can use the solution $p_v = e^{\frac{\theta_v}{n}} + e^{-\frac{\theta_v}{n}} $.

\begin{conjecture}\cite{BWY1} \label{Conj}
    In the above setup, 
    \begin{align*}
        \lim_{n\rightarrow\infty} \frac{1}{n} \log |\tr \; \Lambda_{\varphi,r}^q| = \frac{1}{4\pi} \Vol(M_{\varphi})
    \end{align*}
    where $\Vol(M_{\varphi})$ is the volume of the complete hyperbolic metric of the mapping torus $M_{\varphi} = S \times [0,1]/(x,1) \sim (\varphi(x),0)$.
\end{conjecture}

We verify the conjecture for diffeomorphisms of the four-puncture sphere under a technical assumption that the volume of the mapping torus is large. More precisely, such a mapping torus has a natural ideal triangulation with $2k_0$ ideal tetrahedra, which implies that the volume of its complete hyperbolic metric is bounded by $2k_0 v_3$, where $v_3 \approx 1.0145\dots $ is the volume of the regular ideal tetrahedron. We will require the volume to be at least $2(k_0-1)v_3$. 
 
\begin{theorem} \label{asymptotics}
    Let $\varphi\colon S_{0,4} \rightarrow S_{0,4}$ be an orientation-preserving, pseudo-Anosov diffeomorphism of the four-puncture sphere. 
    Let $[r] \in \mathcal{X}_{\SL(\CC)}(S)$ be a generic $\varphi$-invariant character, with puncture weights $p_1, p_2, p_3, p_4$ that are $\varphi$-invariant. Assume in addition that  $\Vol (M_{\varphi})> 2(k_0 -1)v_3$, where $2k_0$ is the number of tetrahedra appearing in the canonical ideal triangulation of the mapping torus $M_{\varphi}$ and where $v_3$ is the volume of the regular ideal hyperbolic tetrahedron.
    Then,
    \begin{align*}
        \left | \tr \; \Lambda_{\varphi, r}^q \right| = c e^{\frac{n}{4\pi} \Vol (M_{\varphi})} \left( 1 + O \left( \tfrac{1}{n}\right)\right)
    \end{align*}
    where $c \neq 0$ is a constant that does not grow exponentially in n. 
\end{theorem}
\begin{remark}
    There is at least one infinite family of mapping torus for four-puncture spheres that satisfy the condition in the theorem \ref{asymptotics}, the diffeomorphism word is given by $(LR)^k$ for any $k>1$, where the result holds true. This is because for $LR$, there are four tetrahedra with volume is $v_3$ for each of them, and we can stack set of these four tetrahedra together vertically (giving us $(LR)^k$ for stacking k times) to get the desired result. 
\end{remark}
As a Corollary, we prove that the Conjecture \ref{Conj} holds in this case.
\begin{corollary} \label{volume conjecture}
    Under the hypotheses of Theorem {\upshape \ref{asymptotics}},
    \begin{align*}
        \lim_{n\rightarrow\infty} \frac{1}{n} \log |\tr \; \Lambda_{\varphi,r}^q| = \frac{1}{4\pi} \Vol(M_{\varphi}).
    \end{align*}
\end{corollary}

\medskip
\noindent \textbf{Acknowledgements.}
This research was partially supported by the NSF grants DMS-1812008 and DMS-2203334 (PI: Tian Yang). The author would like to thank his supervisor Tian Yang for guidance and support. The author would also like to thank Francis Bonahon and Ka Ho Wong for useful discussions.

\section{Surface diffeomorphisms and the Chekhov-Fock algebra}
\subsection{Surface diffeomorphisms for Four-puncture sphere}
Here, we recall the surface diffeomorphism for four-puncture spheres. 
The mapping class group of the four-puncture sphere is isomorphic to the semi-direct product $\text{PSL}_2(\ZZ) \ltimes (\ZZ_2 \times \ZZ_2) $ \cite{FM}. Consider the generators $L = \begin{pmatrix}
    1 & 1\\
    0 & 1
\end{pmatrix}$ and $R = \begin{pmatrix}
    1 & 0\\
    1 & 1
\end{pmatrix}$ for $\text{PSL}_2(\ZZ)$ and $(\psi_1, \psi_2) \in \ZZ^2 \times \ZZ^2$. 
\\
Note that the elements $L$ and $R$ fix one puncture while $\psi_1, \psi_2$ permutes the punctures. Therefore, we can write any orientation preserving diffeomorphism $\varphi: S_{0,4} \to S_{0,4}$ as 
\begin{align}\label{pseudo-Anosov}
    \varphi = L^{n_1} R^{m_1} \dots L^{n_k} R^{m_k} \psi_1^{\epsilon_1} \psi_2^{\epsilon_2}
\end{align}
where $n_i, m_i \in \ZZ_{>0}$ and $\epsilon_1, \epsilon_2 \in \{0,1\}$. For the diffeomorphism to be pseudo-Anosov, we require $k\geq 1$ \cite{GF}. As explained in the subsequent sections, these generators will play a fundamental role in the computation of the invariant.

\subsection{Ideal triangulation and edge weight system}
We represent the four-puncture sphere as $S_{0,4} = \hat{S}_{0,4} - \{ p_1,p_2,p_3,p_4\}$ where $p_1, p_2, p_3, p_4$ are the four punctures, and $\hat{S}_{0,4}$ is a compact surface such that the ideal triangulation of $S_{0,4}$ is the ideal triangulation of $\hat{S}_{0,4}$ with the punctures as the vertex set. The triangulation has 6 edges and 4 faces. 
\\
Given an ideal triangulation $\tau$, we assign complex numbers $a_i \in \CC^*$ to each edge $\gamma_i$ of $\tau$. This determines a character $[\tilde r] \in \mathcal{X}_{\PSL(\CC)}$ (see \cite{BL07} for details). We get an edge weight system $\bold a = (a_1, \dots, a_6) \in (\CC^*)^6 $. These $a_i$'s are also called the shear bend parameters.
\begin{lemma}\cite{BWY1}
    Given an ideal triangulation $\tau$ , there exists a Zariski-open dense subset $U \subset \mathcal{X}_{\PSL (\CC)}(S_{0,4})$ such that every character $[\tilde r] \in U$ is associated as above to an edge weight system $\bold a \in (\CC^*)^6$ for $\tau$.
\end{lemma}

\subsection{Ideal triangulation sweeps and $\varphi$-invariant characters}
Recall that the edge labels are fixed in the triangulation. The two elementary operations on a triangulation are either an edge re-indexing or a diagonal exchange (see \cite{BWY1}). The edge re-indexing is just a permutation of the edges $\gamma_i$ such that $\gamma'_i = \gamma_{\sigma(i)}$ for some permutation $\sigma$ of the set $\{1, \dots, 6\}$. The diagonal exchange along the edge $\gamma_{i_0}$ is such that the ideal triangulation $\tau$ with edges $\gamma_i$ are replaced with an ideal triangulation $\tau'$ with edges $\gamma'_i$ such that $\gamma'_i = \gamma_i$ when $i \neq i_0$ and $\gamma'_{i_0}$ is the other diagonal of the square formed by the two faces of $\tau$ that are adjacent to $\gamma_{i_0}$. For reference, see figure \ref{fig: diagonal exchange} (the labels are slightly different and would make sense later). None of the edges in the figure are identified.
\\

We now look at how the edge weight system changes. If the triangulation $\tau'$ is obtained by re-indexing by a permutation $\sigma$, then the edge weights $a'_i = a_{\sigma(i)}$ define the same character $[\tilde r] \in \mathcal{X}_{\PSL(\CC)}(S_{0,4}) $. If $\tau'$ is obtained by a diagonal exchange at its $i_0-$th edge, there exists edge weights $a'_i$ for $\tau'$ define the same character $[\tilde r] \in \mathcal{X}_{\PSL(\CC)}(S_{0,4})$ as long as $a_{i_0} \neq -1$.
The formula is defined in \cite{BL07} as follows
\begin{align*}
    a'_i = \begin{cases}
        a_i & \text{if } i \neq i_0, i_1, i_2, i_3, i_4
        \\
        a_{i_0}^{-1} & \text{if } i = i_0
        \\
        a_i(1 + a_{i_0}) & \text{if } i = i_1 \text{ or } i_3
        \\
        a_i(1 + a_{i_0}^{-1})^{-1} & \text{if }  i = i_2 \text{ or } i_4.
    \end{cases}
\end{align*}
Suppose that we can connect the two ideal triangulations $\tau$ and $\varphi(\tau)$ by a finite sequence of ideal triangulations $\tau = \tau^{(0)}, \tau^{(1)}, \dots, \tau^{(k_0)} = \varphi(\tau) $, where each $\tau^{(k)}$ is obtained from $\tau^{(k-1)}$ by an edge re-indexing or a diagonal exchange (such a sequence always exists). This is called an ideal triangulation sweep from $\tau$ to $\tau'$.  

We start with an edge weight system $a^{(0)} =( a^{(0)}_1, \dots, a^{(0)}_6) \in (\CC^*)^6$ for $\tau^{(0)} = \tau $. By using the formulas above, we obtain an edge weight system $a^{(k)} \in (\CC^*)^6 $ for each $\tau^{(k)}$ that all define the same character $[\tilde r] \in \mathcal{X}_{\PSL (\CC)}(S_{0,4})$. This edge weight system is well defined as long as $a_i^{(k-1)} \neq -1$ whenever we have a diagonal exchange at the i-th edge. We then have an edge weight system $a^{(0)}, a^{(1)}, \dots, a^{(k_0)} $ for an ideal triangulation sweep $\tau = \tau^{(0)}, \dots, \tau^{(k_0)} = \varphi(\tau) $. 

By construction, an edge weight system for an ideal triangulation sweep uniquely determines a character $[\tilde r] \in \mathcal{X}_{\PSL (\CC)}(S_{0,4}) $.

\begin{lemma}\cite{BWY1}
    Let $\tau = \tau^{(0)}, \tau^{(1)}, \dots, \tau^{(k_0)} = \varphi(\tau) $ be an ideal triangulation sweep from $\tau$ to $\varphi(\tau)$. If the weight system $a^{(0)}, a^{(1)}, \dots, a^{(k_0)} $ for this sweep is such that $a^{(k_0)} = a^{(0)} $, then the associated character $[\tilde r] \in \mathcal{X}_{\PSL(\CC) }(S_{0,4})$ is fixed by the action of $\varphi$ on $\mathcal{X}_{\PSL(\CC)}(S_{0,4})$. 

    Conversely, there is a Zariski open dense subset $U \subset \mathcal{X}_{\PSL (\CC)}(S_{0,4})$ such that every $\varphi-$invariant character $[\tilde r] \in U$ is associated in this way to an edge weight system $a^{(0)}, \dots, a^{(k_0)} $ with $a^{(k_0)} = a^{(0)} $.  \qed
\end{lemma}

\subsection{Chekhov-Fock Algebra for the four-puncture sphere}
We will work with the Chekhov-Fock (CF) algebra of the four punctured spheres to compute the intertwiner $ \Lambda_{\varphi,r}^q$. By \cite[Theorem~16]{BWY1}, we see that this intertwiner is conjugate to the intertwiner coming from the Kauffman bracket skein algebra. This implies the trace of the intertwiner remains unchanged, and thus we can use the CF algebra intertwiner to get the invariant. 

Let the edges of the four-punctured sphere be labeled as follows:
\begin{figure}[h] 
    \centering
    \includegraphics[width = 0.4\textwidth]{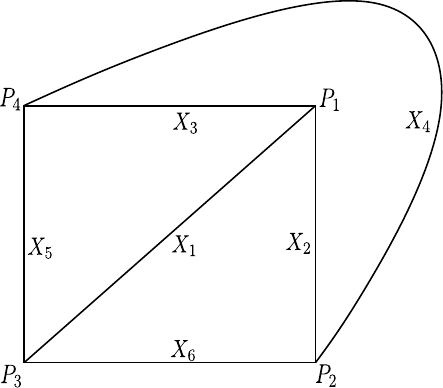}
    \caption{Triangulation for four puncture sphere}
    \label{fig: triangulation S_0,4}
\end{figure}
\\
The Chekhov-Fock algebra is the algebra of the Laurent polynomials in $X_1, \dots, X_6$ associated with the edges of triangulation as in the figure, $\mathcal{T}^q_{\tau} (S_{0,4}) = \CC [X_1^{\pm 1}, \dots, X_6^{\pm 1}] $. The relations for the algebra are given by "$q$-commutativity", that is
\begin{align*}
    X_i X_j = q^{2\sigma_{ij} }X_j X_i,
\end{align*}
where  $ \sigma_{ij} =    \cp{\includegraphics[width=3cm]{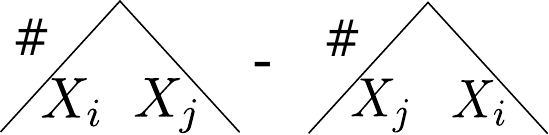}} $.
\\
Thus, the $q$-commuting relations are:
\begin{align*}
    X_1 X_2 = q^{2}X_2 X_1 \;\;\;\;\;\; X_2 X_3 &= q^2 X_3 X_2 \;\;\;\;\;\;\;\; X_3 X_5 = q^{-2} X_5 X_3
    \\
    X_1 X_3 = q^{-2}X_3 X_1 \;\;\;\;\;\; X_2 X_4 &= q^{-2} X_4 X_2 \;\;\;\;\;\; X_4 X_5 = q^2 X_5 X_4
    \\
    X_1 X_5 = q^{2} X_5 X_1 \;\;\;\;\;\; X_2 X_6 &= q^2 X_6 X_2 \;\;\;\;\;\;\;\;X_4 X_6 = q^{-2} X_6 X_4
    \\
    X_1 X_6 = q^{-2} X_6 X_1 \;\;\;\;\;\; X_3 X_4 &= q^{2} X_4 X_3 \;\;\;\;\;\; \;\;X_5 X_6 = q^{2} X_6 X_5.
\end{align*}
See \cite{BL07} for more details on the algebra and its representations.

The classification of irreducible representations of $\mathcal{T}^q_{\tau}(S_{0,4})$ involves central elements $P_j \in \mathcal{T}^q_{\tau}(S_{0,4})$ for $j=1,2,3,4$ associated to each puncture on the surface.
\begin{theorem}\cite[Theorem~21]{BL07}
    If q is a primitive n-th root of unity with n odd,
    and if $\rho\colon \mathcal{T}^q_{\tau}(S_{0,4}) \rightarrow \mathrm{End}(V)$ is a finite-dimensional irreducible representation, then $\dim (V) = n$ and there exists $a_1, a_2, \dots, a_6, p_1, p_2, p_3, p_4 \in \CC^*$ such that
    \begin{align*}
        \rho(X_i^n) = a_i \id_V && \rho(P_j) = p_j \id_V
    \end{align*}
    for every edge $i=1, \dots, 6$ of the ideal triangulation $\tau$ and punctures $j = 1, \dots, 4$.
Furthermore, $\rho$ is determined up to isomorphism by this set of invariants $a_1, \dots, a_6, p_1, \dots, p_4 \in \CC^*$. Such a collection is an irreducible representation if and only if
\begin{align*}
    &p_1^n = a_1 a_2 a_3 & p_2^n = a_2 a_4 a_6 && p_3^n = a_1 a_5 a_6 && p_4^n = a_3 a_4 a_5.
\end{align*}
\qed 
\end{theorem}
We start with a triangulation $\tau^{(0)}$. 
To each edge $e_i$, we will assign $x_i^{(0)} \in \CC^{*}$ in $\tau^{(0)}$. 
The central elements of the Chekhov-Fock algebra corresponding to the four punctures are
\begin{align*}
    P_1 = q^{-1}X_1X_2X_3,
    \;\; 
    P_2 =& q X_2 X_4 X_6,
    \;\;
    P_3 = q^{-1}X_1 X_5 X_6,
    \;\;
    P_4 = q^{-1} X_3 X_4 X_5.
\end{align*}
We also define $H = q^{-2}X_1 X_2X_3X_4X_5X_6$, such that $H^2 = P_1 P_2 P_3 P_4$. 
We have four central elements and six generators, therefore we can choose two independent generators, say $X_1, X_2$. We rewrite the other generators in terms of these two and the central elements. 
\begin{align} \label{dependent edges}
    X_3 = q P_1 X_2^{-1} X_1^{-1}
    \;\;\;\;
    X_4 =  P_4 P_2 H^{-1} X_1
    \;\;\;\;
    X_5 =  H P_1^{-1}P_2^{-1} X_2
    \;\;\;\;
    X_6 = q H P_4^{-1} X_2^{-1} X_1^{-1}.
\end{align}
\begin{remark}
    Here we are actually using the fractional algebra $\hat{\mathcal{T}}_\tau^q(S_{0,4})$. This algebra consists of formal rational fractions in $X_i$ for $i=1, \dots, 6$, modulo the same relations as above.
\end{remark}

\subsection{Standard representation of the algebra $\hat{\mathcal{T}}^q_{\tau} (S_{0,4})$}
The results follow from \cite{BL07}. The results are originally obtained for the Chekov-Fock algebra but can be extended to the fractional algebra by using algebra homomorphisms.
\begin{lemma} 
    A standard representation of the algebra $\hat{\mathcal{T}}^q_{\tau} (S_{0,4})$ associated with non-zero numbers $x, y, p_1, p_2, p_3, p_4  \in \CC^*$ with $\bold p = (p_1,p_2,p_3,p_4) $ is given by a representation $\rho_{x,y,\bold p} \colon \hat{\mathcal{T}}_{\tau}^q(S_{0,4}) \to \mathrm{End}(\CC^n) $ where, if $\{ w_1, \dots, w_n\}$ is the standard basis for $\CC^n$, 
    \begin{align*}
        \rho_{x,y,\bold p}(X_1)(w_i) &= x q^{2i} w_i
        \\
        \rho_{x,y,\bold p} (X_2)(w_i) &= y q^{-i} w_{i+1}
        \\
        \rho_{x,y,\bold p} (X_3) (w_i) &= p_1 (xy)^{-1} q^{-i}w_{i-1}
        \\
        \rho_{x,y,\bold p} (X_4) (w_i) & = x q^{2i} p_4 p_2 h^{-1} w_i
        \\
        \rho_{x,y,\bold p} (X_5) (w_i) & = y q^{-i} h p_1^{-1} p_2^{-1} w_{i+1}
        \\
        \rho_{x,y,\bold p} (X_6) (w_i) &= q^{-i} (xy)^{-1} h p_4^{-1} w_{i-1}
    \end{align*}
    for every $i=1, \dots, n$ and counting indices modulo n.
\end{lemma}
\begin{proof}
    The proof follows from the fact that there are two independent generators $X_1$ and $X_2$ and like the one-puncture torus case, by linear algebra (see \cite{BWY1}), one can see that the representations for $X_1$ and $X_2$ are given by
    \begin{align*}
        \rho_{x,y,\bold p}(X_1) (w_i) = x q^{2i} w_i && \rho_{x,y,\bold p}(X_2) (w_i) = yq^{-i} w_{i+1}.
    \end{align*}
    This implies
    \begin{align*}
        \rho_{x,y,\bold p}(X_1^{-1})(w_i) = x^{-1}q^{-2i} w_i && \rho_{x,y,\bold p}(X_2^{-1})(w_i) = y^{-1} q^{i-1} w_{i-1}.
    \end{align*}
    The central elements are $P_j$ for $j=1,2,3,4$ and $H$, with 
    \begin{align*}
        \rho_{x,y,\bold p}(P_j) = p_j \id_{\CC^n} && \rho_{x,y,\bold p} (H) = h \id_{\CC^n}.
    \end{align*}
    The image for the other $X_j$ for $j=3,4,5,6$ depends on the central elements and therefore can be computed using the homomorphism property of the representation. For example,
    \begin{align*}
    \rho_{x,y,\bold p}(X_3)(w_i) =& \rho_{s} (qP_1 X_2^{-1} X_1^{-1})(w_i) = qp_1 \rho_{x,y,\bold p}(X_2^{-1}) (\rho_{s} (X_1^{-1})(w_i)) 
    \\
    =& qp_1 x^{-1}  q^{-2i} \rho_{x,y,\bold p} (X_2^{-1})(w_i) 
    = y^{-1}p_1 x^{-1}q^{1-2i} q^{i-1}(w_{i-1}) = (xy)^{-1} p_1 q^{-i} w_{i-1}.
\end{align*}
The other representations are calculated similarly. 

Note that $\rho_{x,y,\bold p}(X_1^n) = x^n \id_{\CC^n}, \rho_{x,y,\bold p}(X_2^n) = y^n \id_{\CC^n}$ and $\rho_{x,y,\bold p}(P_j) = p_j \id_{\CC^n}$. We let $a = x^n$ and $b = y^n$.
\end{proof}

\begin{proposition}\cite{BL07} \label{irrep CF}
    \begin{enumerate}
        \item Every standard representation $\rho_{x,y,\bold p}\colon \hat{\mathcal{T}}^q_{\tau}(S_{0,4}) \rightarrow \mathrm{End}(V)$ is irreducible.
        \item Every irreducible representation of $\hat{\mathcal{T}}^q_{\tau}(S_{0,4})$ is isomorphic to a standard representation $\rho_{x,y,\bold p}$.
        \item Two standard representations $\rho_{x,y,\bold p}$ and $\rho_{x',y',\bold p'}$ are isomorphic if and only if $x^n = {x'}^n, y^n = y'^n$ and
        $\bold p = \bold p'$.
    \end{enumerate}
    \qed
\end{proposition}
The above proposition implies that any irreducible representation of the Chekhov-Fock algebra is classified up to isomorphism by the numbers $a,b, p_1, p_2, p_3, p_4, h \in \CC^*$.

\subsection{Coordinate Change isomorphisms}
We study the coordinate change isomorphism of the Chekhov-Fock algebra for any two ideal triangulations $\tau$ and $\tau'$ as $\Phi_{\tau \tau'}^q: \hat{\mathcal{T}}^q_{\tau'} (S_{0,4}) \to \hat{\mathcal{T}}^q_{\tau}(S_{0,4}) $. This isomorphism satisfies the fundamental relation
\begin{align*}
    \Phi_{\tau \tau''}^q = \Phi_{\tau \tau'}^q \circ \Phi_{\tau' \tau''}^q
\end{align*}
for any three ideal triangulations $\tau, \tau', \tau''$. These isomorphisms are given by explicit formulas \cite{BL07}. We consider the triangulation sweep from $\tau$ to $\varphi(\tau)$ and write down the explicit isomorphisms.

If two ideal triangulations differ by an edge re-indexing, that is, if $\tau'$ is obtained from $\tau$ by an edge re-indexing with the i-th edge of $\tau'$ being the $\sigma(i)$-th edge of $\tau$ for some permutation $\sigma$ of $\{1, 2, \dots, 6 \} $, then $\Phi_{\tau \tau'}^q : \hat{\mathcal{T}}^q_{\tau'}(S_{0,4}) \to \hat{\mathcal{T}}^q_{\tau}(S_{0,4}) $ is the unique isomorphism given by $\Phi_{\tau, \tau'}^q(X_i') = X_{\sigma(i)}$.
\begin{figure}[h]
    \centering
    \includegraphics[width = 0.5\textwidth]{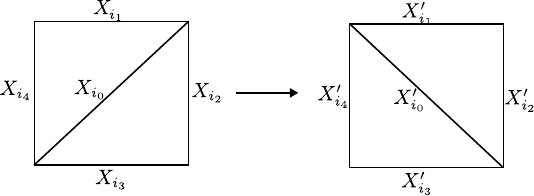}
    \caption{A diagonal exchange}
    \label{fig: original diagonal exchange}
\end{figure}

If $\tau'$ is obtained from $\tau$ using a diagonal exchange as in figure \ref{fig: original diagonal exchange}, with the edge re-indexing, then $\Phi_{\tau \tau'}^q: \hat{\mathcal{T}}^q_{\tau'}(S_{0,4}) \to \hat{\mathcal{T}}^q_{\tau}(S_{0,4}) $ is the unique algebra homomorphism given by 
\begin{align} \label{diagonal exchange}
    \Phi_{\tau \tau'}^q (X_i') = \begin{cases}
        X_i & \text{ if } i \neq i_0, i_1, i_2, i_3, i_4
        \\
        X_{i_0}^{-1} & \text{ if } i = i_0
        \\
        X_i(1 + q^{-1}X_{i_0}) & \text{ if } i = i_1 \text{ or } i_3
        \\
        X_i(1 + q^{-1}X_{i_0}^{-1})^{-1} & \text{ if } i = i_2 \text{ or } i_4.
    \end{cases}
\end{align}
We now consider an algebra homomorphism $\rho \circ \Phi^q_{\tau \tau'} \colon \hat{\mathcal{T}}^q_{\tau'}(S_{0,4}) \to \mathrm{End}(V) $. For more details on the algebra homomorphism, see \cite[Section~3.4]{BWY1}.

\section{The Quantum Invariant}
\subsection{Isomorphisms of the Chekhov-Fock algebra}
We define the invariant using an ideal triangulation sweep from an ideal triangulation $\tau$ to an ideal triangulation $\varphi(\tau)$ of the four-puncture sphere bundle. We have seen that any pseudo-Anosov and orientation-preserving diffeomorphism (\ref{pseudo-Anosov}) can be written as 
\begin{align*}
    \varphi = L^{n_1} R^{m_1} \dots L^{n_k} R^{m_k} \psi_1^{\epsilon_1} \psi_2^{\epsilon_2}
\end{align*}
where $n_i, m_i \in \ZZ_{>0}$ and $\epsilon_1, \epsilon_2 \in \{0,1\}$. 
We  rewrite the diffeomorphism $\varphi$ as
\begin{align} \label{diffeo}
    \varphi = \varphi_1 \circ \dots \circ \varphi_{k_0} \circ \psi_1^{\epsilon_1} \psi_2^{\epsilon_2}
\end{align}
where each $\varphi_k$ corresponds to $L$ or $R$ for $k=\{1,\dots, k_0\}$. We start with a triangulation $\tau_0$ of the surface with independent edges labelled by $(X_1, X_2)$ and shear-bend parameters $(a_0, b_0) \in (\CC^*)^2$.

For each $k = 1, \dots, k_0$, we define the triangulation $\tau_k = \varphi_1 \circ \dots \circ \varphi_{k} (\tau_0)$. 
\begin{figure}[h]
        \centering
        \includegraphics[width = 0.45\textwidth]{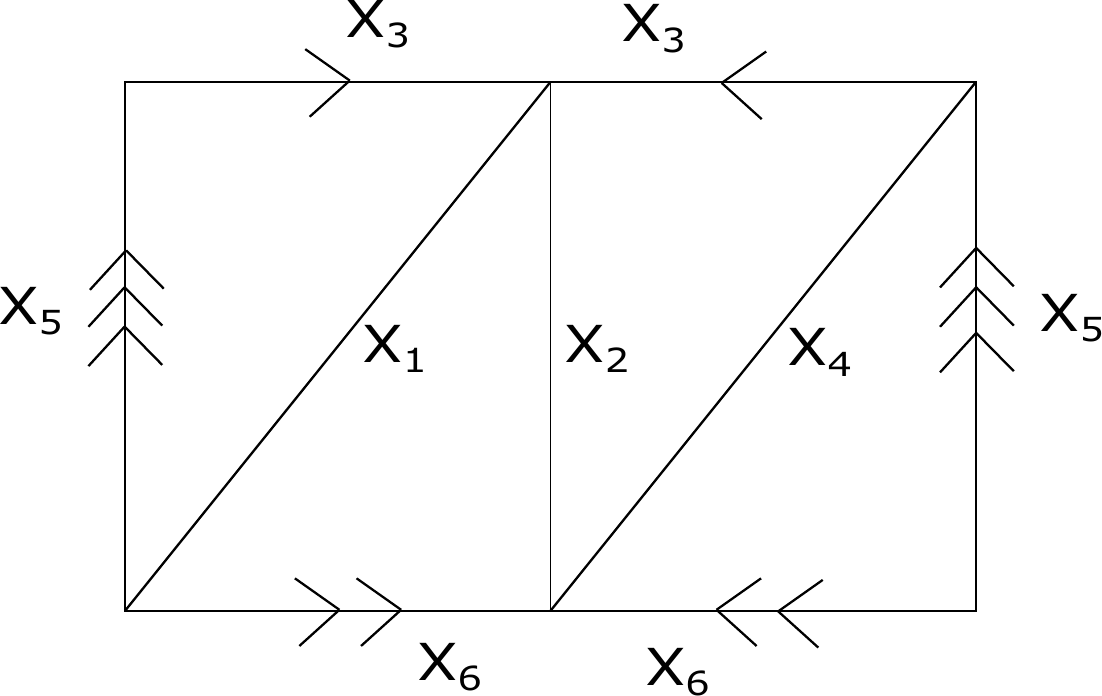}
        \caption{We do diagonal exchange along $X_2, X_5$ for L and $X_3, X_6$ for R}
        \label{fig: diagonal exchange}
\end{figure}
\\
We start with the triangulation (figure \ref{fig: diagonal exchange}). For each generator L or R, of the four-punctured sphere, we use two tetrahedra for diagonal exchange.
The edges in the diagonal exchanges are $X_3, X_6$ for \textit{R}, and $X_2, X_5$ for \textit{L}.
\begin{figure}[h]
    \centering
    \includegraphics[width = 0.75\textwidth]{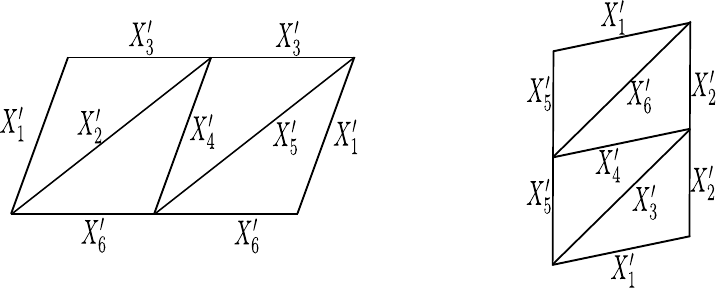}
    \caption{(Left) Diagonal switch corresponding to the isomorphism L. (Right) Diagonal switch corresponding to the isomorphism R}
    \label{fig: left and right switch}
\end{figure}
\\
We now define isomorphisms $\hat{\mathcal{T}}^q_{\tau} \to \hat{\mathcal{T}}^q_{\tau}$ to study the algebra homomorphisms of the representations under this isomorphism. 
\begin{proposition}
The left isomorphism $\mathcal{L} \colon \hat{\mathcal{T}}^q_{\tau}(S_{0,4}) \to \hat{\mathcal{T}}^q_{\tau}(S_{0,4}) $ is defined by
\begin{align*}
    &\mathcal{L}(X_1) = X_2^{-1} & \mathcal{L}(X_4) &= X_5^{-1} 
    \\
    &\mathcal{L}(X_2) = (1 + qX_2)(1 + qX_5)X_4 & \mathcal{L}(X_5) &= (1 + qX_2)(1 + qX_5)X_1
    \\
    &\mathcal{L}(X_3) = (1 + qX_2^{-1})^{-1}(1 + qX_5^{-1})^{-1} X_3 & \mathcal{L}(X_6) &= (1 + qX_2^{-1})^{-1}(1 + qX_5^{-1})^{-1}X_6.
\end{align*}
The right isomorphism $\mathcal{R} \colon \hat{\mathcal{T}}^q_{\tau}(S_{0,4}) \to \hat{\mathcal{T}}^q_{\tau}(S_{0,4}) $
\begin{align*}
    &\mathcal{R}(X_1) = X_3^{-1} &\mathcal{R}(X_4) &= X_6^{-1} 
    \\
    &\mathcal{R}(X_2) = (1 + qX_3)(1 + qX_6)X_2  & \mathcal{R}(X_5) &= (1 + qX_3)(1 + qX_6)X_5 
    \\
    &\mathcal{R}(X_3) = (1 + qX_3^{-1})^{-1} (1 + qX_6^{-1})^{-1}X_4& \mathcal{R}(X_6) &= (1 + qX_3^{-1})(1 + qX_6^{-1})^{-1}X_1.
\end{align*}

\end{proposition}
\begin{proof}
    Let's look at left isomorphism first. The right isomorphism follows the same idea.

    In figure \ref{fig: left and right switch}, $X'_2 = X_2^{-1}, X'_5 = X_5^{-1}$ because the diagonal exchange is done along those two edges. 
    From eq \ref{diagonal exchange}, $X'_4$ first becomes $(1 + qX_5)X_4$ through the diagonal exchange along $X_5$ and after the diagonal exchange along $X_2$, it becomes $(1 + qX_2) (1 + qX_5) X_4$. Similarly, the other edge changes are captured. 
    
    Now, we look at the re-indexing. We want to compare the changed (with diagonal switch) triangulation to have the same "$q$-commuting" relations, therefore we re-index the edges as in figure \ref{fig: diagonal exchange}. Therefore, $\mathcal{L}(X_5)$ becomes $X'_1$, $\mathcal{L}(X_1)$ becomes $X'_2$, $\mathcal{L}(X_2)$ becomes $X'_4$ and so on.

    Note that the order of the diagonal switch is important, or else we will get an additional factor of q. 
\end{proof}
We want to determine the representations $\rho_{x_1,y_1,\bold p^{(k)}}\circ \mathcal{L} $ and $\rho_{x_1,y_1, \bold p^{(k)}} \circ \mathcal{R}$ of $\hat{\mathcal{T}}^q_{\tau^{(k)}}$ for a standard representation $\rho_{x_1,y_1,\bold p^{(k)}} \colon \hat{\mathcal{T}}_{\tau^{(k)}}^q(S_{0,4}) \to \mathrm{End}(V)$.

\begin{proposition}\label{shear-bend_relation}
    Let $x_{k-1}, y_{k-1} \in \CC^*$ and $\bold p^{(k-1)} \in (\CC^*)^4$ be given. 
    \\
    If $y_{k-1}^n \neq -1$ and $h^n y_{k-1}^n \neq - (p^{(k-1)}_1)^n (p^{(k-1)}_2)^n$, choose $u_k,\hat{u}_k, v_k, \hat{v}_k \in \CC^*$ such that
    \begin{align*}
        &u_k = qy_{k-1} & \hat{u}_k = qh(p^{k-1}_1)^{-1} (p^{(k-1)}_2)^{-1}y_{k-1} && v_k^n = 1+ u_k^n && \hat{v}_k^n = 1 + \hat{u}_k^n.
    \end{align*}
    Then the representation $\rho_{x_{k-1},y_{k-1},\bold p^{(k-1)}} \circ \mathcal{L}$ is isomorphic to a standard representation $\rho_{x_k,y_k, \bold p^{(k)}}$ with
    \begin{align*} 
         &x_k =  y_{k-1}^{-1} \;\;\;\;\;\;\;\;\;\;\;\; y_k = v_k \hat{v}_k p^{(k-1)}_2 p^{(k-1)}_4 h^{-1} x_{k-1} 
         \\
         &p^{(k)}_1 = p^{(k-1)}_4 \;\;\;\; p^{(k)}_2 = p^{(k-1)}_2 \;\;\;\; p^{(k)}_3 = p^{(k-1)}_3 \;\;\;\; p^{(k)}_4 = p^{(k-1)}_1. 
    \end{align*}
    Similarly, if $(p^{(k-1)}_1)^n \neq -x_{k-1}^n y_{k-1}^n$ and $(p^{(k-1)}_4)^n x_{k-1}^n y_{k-1}^n \neq -h^n$, choose $u_k,\hat{u}_k, v_k, \hat{v}_k \in \CC^*$ such that
    \begin{align*}
        &u_k = q p^{(k-1)}_1 x_{k-1}^{-1} y_{k-1}^{-1} & \hat{u}_k = q h (p^{(k-1)}_4)^{-1}x_{k-1}^{-1}y_{k-1}^{-1} && v_k^n = 1 + u_k^n && \hat{v}_k^n = 1 + \hat{u}_k^n.
    \end{align*}
    Then the representation $\rho_{x_{k-1},y_{k-1}, \bold p^{(k-1)}} \circ \mathcal{R}$ is isomorphic to a standard representation $\rho_{x_k, y_k, \bold p^{(k)}}$ with
    \begin{align*}
        &x_k = (p^{(k-1)}_1)^{-1} x_{k-1} y_{k-1} \;\;\;\;\;\;\;\;\;\;\;\;\;\;\;\; y_k = v_k \hat{v}_k y_{k-1}
        \\
        &p^{(k)}_1 = p^{(k-1)}_2 \;\;\;\; p^{(k)}_2 = p^{(k-1)}_1 \;\;\;\; p^{(k)}_3 = p^{(k-1)}_3 \;\;\;\; p^{(k)}_4 = p^{(k-1)}_4.
    \end{align*}
\end{proposition}
\begin{proof}
    The computation can be found in \cite[Lemma 27]{BL07} corresponding to the isomorphisms defined in the previous Proposition \ref{shear-bend_relation}.
\end{proof}
\begin{remark}
    The puncture weight $p_3$ does not change in either $\mathcal{L}$ or $\mathcal{R}$.
\end{remark}
The shear bend parameters $(a,b) \in (\CC^*)^2$ are just n-th powers of $(x,y) \in (\CC^*)^2$. Therefore, under the two isomorphisms, for the choices $u_k, v_k, \hat{u}_k, \hat{v}_k \in \CC^*$, the shear bend parameters change in the following manner
\begin{align}\label{edge_weight_relation}
    \mathcal{L} \colon &a_k =  b_{k-1}^{-1} && b_k = v_k^n \hat{v}_k^n (p^{(k-1)}_2)^n (p^{(k-1)}_4)^n h^{-n} a_{k-1}; \notag
    \\
    \mathcal{R} \colon &a_k = (p^{(k-1)}_1)^{-n} a_{k-1} b_{k-1} && b_k = v_k^n \hat{v}_k^n b_{k-1}.
\end{align}

\subsection{Intertwiners}
Let $\tau = \tau^{(0)}, \tau^{(1)}, \dots, \tau^{(k_0)} = \varphi(\tau) $ be an ideal triangulation sweep connecting $\tau$ and $\varphi (\tau)$ and let $a = a^{(0)}, a^{(1)}, \dots, a^{(k_0)} = a $ be a periodic edge weight system for this ideal triangulation sweep, defining a $\varphi$-invariant character $[\tilde r] \in \mathcal{X}_{\PSL (\CC)}(S_{0,4})$. 
The classification of irreducible representations involves, in addition to an edge weight system $a \in (\CC^*)^6$ for $\tau$, puncture weights $p_j \in \CC^*$ for $j = 1,2,3,4$ (see \cite[Section~3.4]{BWY1}). We choose the puncture weights such that they are $\varphi$-invariant. 

Let $\rho: \hat{ \mathcal{T}}_{\tau}^q (S_{0,4}) \to \mathrm{End}(V)$ be an irreducible representation of the (fractional) Chekhov-Fock algebra of $\tau$ that is classified by the edge weight system $a \in (\CC^*)^6$ and $\varphi$-invariant puncture weights.
The representation $\rho \circ \Phi_{\tau \varphi(\tau)}^q \colon \hat{\mathcal{T}}^q_{\varphi(\tau)}(S_{0,4}) \to \mathrm{End}(V) $ is classified by the same edge weight system and puncture weights as $\rho$.
The diffeomorphism $\varphi$ induces a natural isomorphism $\psi^q_{\varphi(\tau) \tau } \colon \hat{\mathcal{T}}^q_{\tau}(S_{0,4}) \to \hat{\mathcal{T}}^q_{\varphi(\tau)} (S_{0,4})$ sending the generators of $\hat{\mathcal{T}}^q_{\tau} (S_{0,4})$ corresponding to the i-th edge of $\tau$ to the generator of $\hat{\mathcal{T}}^q_{\varphi(\tau)}(S_{0,4})$ corresponding to the i-th edge of $\varphi(\tau)$. 
The relabelling is done in a consistent way such that the generators always satisfy the relation $X_i X_j = q^{2\sigma_{ij}} X_j X_i$ in both $\hat{\mathcal{T}}^q_{\tau}(S_{0,4})$ and $\hat{\mathcal{T}}^q_{\varphi(\tau)}(S_{0,4}) $.

We consider the representation $\rho \circ \Phi_{\tau \varphi(\tau)}^q \circ \Psi^q_{\varphi(\tau) \tau} \colon \hat{\mathcal{T}}^q_{\tau}(S_{0,4}) \to \mathrm{End}(V) $. 
Its invariants are the same edge weight system $a \in (\CC^*)^6$ as $\rho \circ \Phi_{\tau \varphi(\tau)}^q $. This implies that the representations $\rho$ and $\rho \circ \Phi_{\tau \varphi(\tau)}^q \circ \Psi^q_{\varphi(\tau) \tau} \colon \hat{\mathcal{T}}^q_{\tau} \to \mathrm{End}(V)$ are isomorphic, by an isomorphism $\Lambda_{\varphi, \tilde r}^q \colon V \to V$ such that
\begin{align}
    \left( \rho \circ \Phi_{\tau \varphi(\tau)}^q \circ \Psi^q_{\varphi(\tau) \tau} \right)(X) = \Lambda_{\varphi, \tilde r}^q \circ \rho(X) \circ {\Lambda_{\varphi, \tilde r}^q}^{-1} \in \mathrm{End}(V)
\end{align}
for every $X \in \hat{\mathcal{T}}^q_{\tau}(S_{0,4})$. We normalize it such that the determinant $\det \Lambda_{\varphi, \tilde r}^q =1$.
\begin{proposition}\cite[Proposition~15]{BWY1}
    Let $\Lambda_{\varphi,\tilde r}^q \colon V \to V$ be the normalized intertwiner as above. Then, up to conjugation, $\Lambda_{\varphi, \tilde r}^q$ depends only on the diffeomorphism $\varphi$, the ideal triangulation sweep from $\tau $ to $\varphi (\tau)$, the periodic edge weight system for this sweep and the $\varphi$-invariant puncture weights.
    \\
    In particular, $|\tr \Lambda_{\varphi, \tilde r}^q|$ is uniquely determined by this data.
\end{proposition}
We compute the intertwiners for each $\varphi_k$ in the diffeomorphism $\varphi$ \ref{diffeo}. 
\\
Consider the discrete quantum dilogarithm function 
\begin{align*}
    QDL^q(u,v|j) = v^{-j} \prod_{k=1}^j (1 + uq^{-2k})
\end{align*}
and
\begin{align*}
    D^q(u) = \prod_{j=1}^n QDL^q(u,v|j).
\end{align*}
The left and right intertwiners are the isomorphisms that satisfy
\begin{align*}
    (\rho_{x_{k-1},y_{k-1}, \bold p^{(k-1)}} \circ \mathcal{L} ) (W) = (\Lambda_L)_{u_k,v_k,\hat{u}_k,\hat{v}_k} \;\; \rho_{x_k,y_k,\bold p^{(k)}} \;\; (\Lambda_L)_{u_k,v_k,\hat{u}_k,\hat{v}_k}^{-1} 
    \\
    (\rho_{x_{k-1},y_{k-1}, \bold p^{(k-1)}} \circ \mathcal{R} ) (W) = (\Lambda_R)_{u_k,v_k,\hat{u}_k,\hat{v}_k} \;\;  \rho_{x_k,y_k,\bold p^{(k)}} \;\; (\Lambda_R)_{u_k,v_k,\hat{u}_k,\hat{v}_k}^{-1}
\end{align*}
for every $W \in \hat{\mathcal{T}}^q_{\tau}(S_{0,4})$. 
\begin{proposition}
    The left and right intertwiners are given by
    \begin{align*}
    ((\Lambda_L)_{u_k, v_k, \hat{u}_k, \hat{v}_k})_{i,j} = \frac{q^{(j^2 - i^2 + 4ij +i - j )/2} QDL^q(u_k, v_k|j) QDL(\hat{u}_k, \hat{v}_k|j) }{\sqrt{n} D^q(u_k)^{1/n} D(\hat{u}_k)^{1/n} }
    \\
    ((\Lambda_R)_{u_k, v_k, \hat{u}_k, \hat{v}_k})_{i,j} = \frac{q^{(3j^2 + i^2 - 4ij +i - j )/2} QDL^q(u_k, v_k|j) QDL(\hat{u}_k, \hat{v}_k|j) }{\sqrt{n} D^q(u_k)^{1/n} D(\hat{u}_k)^{1/n}  }.
\end{align*}
\end{proposition}
\begin{proof}
    For isomorphism $(\Lambda_L)_{u_k,v_k, \hat{u}_k,\hat{v}_k}$, suppose $y_{k-1}^n \neq -1$. We have $u_k,v_k,\hat{u}_k, \hat{v}_k \in \CC$, such that 
    \begin{align*}
        u_k = q x_k^{-1},\;\;\; v_k^n = 1 + u_k^n, \;\;\;\hat{u}_k = q\frac{p_3 p_4}{h}y_{k-1},\;\;\; \hat{v}_k^n = 1 + \hat u_k^n.
    \end{align*}
    We do the computation for the $\mathcal{L}$ case first. 
    Evaluation when $W = X_1$
\begin{align*}
    \rho_{x_{k-1},y_{k-1}, \bold p^{(k-1)}}\circ \mathcal{L}(X_1) \;\; (\Lambda_L)_{u_k, v_k, \hat{u}_k, \hat{v}_k} (w_j) 
    = \rho_{x_{k-1},y_{k-1}, \bold p^{(k-1)}} ( X_2^{-1}) \sum_i ((\Lambda_L)_{u_k, v_k, \hat{u}_k, \hat{v}_k})_{i,j} w_i
    \\
    = \frac{1}{y_{k-1}} \sum_i ((\Lambda_L)_{u_k, v_k, \hat{u}_k, \hat{v}_k})_{i,j} q^{i-1}w_{i-1}
    = \frac{1}{y_{k-1}} \sum_i ((\Lambda_L)_{u_k, v_k, \hat{u}_k, \hat{v}_k})_{i+1,j} q^{i} w_{i}.
\end{align*}
Now other side of the equation:
\begin{align*}
    (\Lambda_L)_{u_k, v_k, \hat{u}_k, \hat{v}_k} \;\; \rho_{x_k,y_k, \bold p^{(k)}} (X_1)(w_j) = (\Lambda_L)_{u_k,\hat{u}_k} x_k q^{2j} w_j = x_k q^{2j}\sum_i ((\Lambda_L)_{u_k, v_k, \hat{u}_k, \hat{v}_k})_{i,j} w_i.
\end{align*}
From the above two equations we get:
\begin{align}
    ((\Lambda_L)_{u_k, v_k, \hat{u}_k, \hat{v}_k})_{i+1,j} = x_k y_{k-1} q^{2j-i} ((\Lambda_L)_{u_k, v_k, \hat{u}_k, \hat{v}_k})_{i,j} \label{eq L i+1}
    \\
    ((\Lambda_L)_{u_k, v_k, \hat{u}_k, \hat{v}_k})_{i-1,j} = x_k^{-1} y_{k-1}^{-1}q^{i-2j-1} ((\Lambda_L)_{u_k, v_k, \hat{u}_k, \hat{v}_k})_{i,j}. \label{eq L i-1}
\end{align}
Evaluation with $W = X_2$:
\begin{align*}
    &\rho_{x_{k-1},y_{k-1}, \bold p^{(k-1)}} \circ \mathcal{L}(X_2) \cdot (\Lambda_L)_{u_k, v_k, \hat{u}_k, \hat{v}_k} (w_j) 
    \\
    &= \rho_{x_{k-1},y_{k-1}, \bold p^{(k-1)}} \left((1 + qX_2)\left(1 + qX_5 \right)X_4\right) \sum_i ((\Lambda_L)_{u_k, v_k, \hat{u}_k, \hat{v}_k})_{i,j} w_i
    \\
    &=\frac{p_2 p_4 x_{k-1}}{h} \rho_{x_{k-1},y_{k-1}, \bold p^{(k-1)}} \left(( 1 +qX_2) \left( 1 + qH^{-1}P_3P_4X_2\right)\right) \sum_i q^{2i} ((\Lambda_L)_{u_k, v_k, \hat{u}_k, \hat{v}_k})_{i,j}w_i
    \\
    &= \frac{p_2 p_4 x_{k-1}}{h}\left( 1 + \frac{p_3 p_4 q^{-1-2j}}{h x_k} \right) \rho_{x_{k-1}, y_{k-1}, \bold p^{(k-1)}}(1 + qX_2)\sum_i q^{2i} ((\Lambda_L)_{u_k, v_k, \hat{u}_k, \hat{v}_k})_{i,j} w_i
    \\
    &= \sum_i \frac{p_2 p_4 x_{k-1}}{h} q^{2i}\left( 1 +
    \frac{p_3 p_4 q^{-1-2j}}{h x_k } \right) \left( 1 + \frac{q^{-1-2j}}{x_k}\right) ((\Lambda_L)_{u_k, v_k, \hat{u}_k, \hat{v}_k})_{i,j} w_i
\end{align*}
where we used (\ref{eq L i+1}) and (\ref{eq L i-1}) multiple times.
\begin{align*}
    (\Lambda_L)_{u_k, v_k, \hat{u}_k, \hat{v}_k} \cdot \rho_{x_k, y_k, \bold p^{(k)}}(X_2)(w_j) &= (\Lambda_L)_{u_k, v_k, \hat{u}_k, \hat{v}_k} y_k q^{-j} w_{j+1} 
    \\
    &= y_k q^{-j}\sum_i ((\Lambda_L)_{u_k, v_k, \hat{u}_k, \hat{v}_k})_{i,j+1} w_i.
\end{align*}
Thus,
\begin{align} \label{eq L j+1}
    ((\Lambda_L)_{u_k, v_k, \hat{u}_k, \hat{v}_k})_{i,j+1} = \frac{p_2 p_4 x_{k-1}q^{2i+j}\left(1 +\frac{q^{-1-2j}}{x_k} \right)\left(1 + \frac{hq^{-1-2j}}{x_kp_3p_4}\right)}{hy_k} ((\Lambda_L)_{u_k, v_k, \hat{u}_k, \hat{v}_k})_{i,j}.
\end{align} 
Choose 
$$((\Lambda_L)_{u_k, v_k, \hat{u}_k, \hat{v}_k})_{0,0} = \frac{1}{\sqrt{n} D^q(u_k)^{1/n} D(\hat{u}_k)^{1/n} } \;.$$
By (\ref{eq L i+1}) and (\ref{eq L j+1}), we get:
\begin{align*}
    ((\Lambda_L)_{u_k, v_k, \hat{u}_k, \hat{v}_k})_{i,0} &=  q^{-i(i-1)/2}
    \\
    ((\Lambda_L)_{u_k, v_k, \hat{u}_k, \hat{v}_k})_{i,j} &= \left(\frac{p_2 p_4 x_{k-1}q^{2i}}{h y_k}\right)^j \left(q^{j(j-1)/2}\prod_{l=0}^{j-1} \left(1 + \frac{q^{-1-2l}}{x_k} \right) \left(1+\frac{hq^{-1-2l}}{x_k p_3p_4} \right)  \right)
    \\
    &\times ((\Lambda_L)_{u_k, v_k, \hat{u}_k, \hat{v}_k})_{i,0}
    \\
    &=  \left(\frac{p_2 p_4 x_{k-1} q^{2i}}{h y_k}\right)^j \left(q^{j(j-1)/2}\prod_{l=1}^{j} \left(1 + \frac{p_1p_2 y_{k-1} q^{1-2l}}{h} \right) \left(1+y_{k-1} q^{1-2l} \right)  \right)
    \\
    &\times ((\Lambda_L)_{u_k, v_k, \hat{u}_k, \hat{v}_k})_{i,0}
    \\
    & = q^{\frac{j^2-i^2+4ij+i-j}{2}} QDL^q( u_k, v_k|j) QDL^q(\hat{u}_k, \hat{v}_k|j)\frac{1}{\sqrt{n} D^q(u_k)^{1/n} D(\hat{u}_k)^{1/n} }\;.
\end{align*}
    A similar computation with $W = X_1$ and $X_2$ gives the result for $\mathcal{R}$.
\end{proof}
\begin{remark}
    For the generators $\psi_1$ and $\psi_2$, the isomorphic representations correspond to an edge-re-indexing. It's easy to check that the intertwiner is trivial for any edge re-indexing. These generators add to the periodic edge weight system up to two tuples $(a_1^{(k_0+1)}, \dots, a_6^{(k_0+1)})$ and $(a_1^{(k_0+2)}, \dots, a_6^{(k_0+2)})$, that are related to the last one $(a_1^{(k_0)}, \dots, a_6^{(k_0)})$ by a permutation. We will ignore the (possible) additional permutation for proof purposes. We add a remark every time this permutation comes up, but it will not affect the proof.
\end{remark}

For computation, we will work with the diffeomorphism where the first generator is $L$. We have a periodic edge weight system, that is, $(a_0, b_0) = (a_{k_0}, b_{k_0})$.  
\begin{proposition}
    For the mapping torus, with the identification of $(x,1) \sim (\phi(x),0) $ where $\phi\colon \Sigma \rightarrow \Sigma$ is the surface diffeomorphism, we get a twist intertwiner $T\colon \CC^n \rightarrow \CC^n$ such that $\rho_{x_0,y_0, \bold p^{(0)}} (W) = T \circ \rho_{x_{k_0},y_{k_0}, \bold p^{(k_0)}} (W) \circ T^{-1}$ for every $W \in \hat{\mathcal{T}}^q_{\tau}(S_{0,4})$. Then,
\begin{align}
    T_{i,j} = q^{-2j l_2 - j l_1} \delta_{i,j+l_1}
\end{align}
where the indices are modulo n and 
$x_{0} = q^{2l_1} x_{k_0}$ and $y_0 = q^{2l_2} y_{k_0}$.
\end{proposition}
\begin{proof}
    Follows from direct calculation.
\end{proof}

\subsection{Logarithm Choices}
For the edge weight system of the triangulation sweep, we need to make choices of logarithms to make things well-defined.

We fix complex numbers $A_k, B_k \in \CC$, such that $a_k = e^{A_k}, b_k = e^{B_k}$. Fix $\theta_j \in \CC$, such that we have Trace$(\alpha_j) = -e^{\theta_j} - e^{-\theta_j} = -p_j^n - p_j^{-n}$ where $\alpha_j \in \pi_1(S_{0,4})$ is the peripheral curve around the puncture and $p_j \in \CC$ is the $\varphi-$invariant puncture weight labelled according to an ideal triangulation $\tau$. We therefore have $e^{\theta_j} = p_j^n$ for $j \in \{1,2,3,4\}$. We fix $h^{n} = e^{(\theta_1 + \theta_2 + \theta_3 + \theta_4)/2}.$
\\
For each $k= \{1, \dots, k_0\}$, we fix $U_k,V_k, \hat{U}_k, \hat{V}_k \in \CC$ such that 
\begin{align*}
    &e^{U_k} = a_k^{-1}
    &&e^{V_k} =1 + e^{U_k}
    \\
    &e^{\hat{U}_k} =
    \begin{cases}
         h^{-n} (p^{(k-1)}_1)^n (p^{(k-1)}_2)^n a_k^{-1} & \text{ if } \varphi_k = L
        \\
        h^n (p^{(k-1)}_1)^{-n} (p^{(k-1)}_4)^{-n} a_k^{-1} & \text{ if } \varphi_k = R
    \end{cases}
    && e^{\hat{V}_k} = 1 + e^{\hat{V}_k}.
\end{align*}
From \ref{edge_weight_relation}, we define
\begin{align*} 
    &\mathcal{L}\colon A_k = - B_{k-1} && B_k = V_k + \tilde{V}_k + (\theta_2^{(k-1)} + \theta_4^{(k-1)} - \theta_1^{(k-1)} - \theta_3^{(k-1)})/2 +  A_{k-1}
    \\
    & \mathcal{R}\colon A_k = -\theta^{(k-1)}_1 + A_{k-1} + B_{k-1} && B_k = V_k + \tilde{V}_k + B_{k-1}.
\end{align*}
Since the edge weight system is periodic, we get
\begin{align*}
    e^{A_{k_0}} = e^{A_0} && e^{B_{k_0}} = e^{B_0}.
\end{align*}
This implies there exists integers $\hat{l}_1, \hat{l}_2 \in \ZZ $ such that 
\begin{align*}
    A_0 = A_{k_0} + 2\pi i \hat{l}_1 && B_0 = B_{k_0} + 2\pi i \hat{l}_2.   
\end{align*}
Note that we need $x_0 = q^{2{l}_1}x_{k_0}$ and $y_0 = q^{2{l}_2}y_{k_0}$. This implies
\begin{align*}
    e^{\frac{A_0}{n}} = e^{\frac{A_{k_0}}{n} + \frac{2\pi i \hat{l}_1}{n}} = x_{k_0} q^{2{l}_1} && e^{\frac{B_0}{n}} = e^{\frac{B_{k_0}}{n} + \frac{2\pi i \hat{l}_2}{n}} = y_{k_0} q^{2{l}_2}.
\end{align*}
Thus, 
\begin{align*}
    l_1 = \hat{l}_1 \frac{(n-1)^2}{2} && l_2 = \hat{l}_2 \frac{(n-1)^2}{2}.
\end{align*}

\subsection{Explicit formula for the invariant}
\begin{proposition}
    For the intertwiner $\Lambda_{\varphi,\tilde r}^q$ and $u_k,\hat{u}_k,v_k,\hat{v}_k \in \CC$ defined above,
    \begin{align} \label{alg formula}
    \tr (\Lambda_{\varphi,\tilde r}^q) =\text{A}_n \sum_{i_1, \dots, i_{k_0}=1}^n &\exp{\left(\frac{2\pi i}{n} \left(-i_{k_0} \hat{l}_2 - i_1 \hat{l}_1 + \sum_{k=1}^{k_0} i_{l+1}^2 + \frac{\epsilon_{k+1}}{2}\left( i_l^2 + i_{l+1}^2 -4i_l i_{l+1} \right) \right) \right)} \notag   
    \\
    &\times \prod_{k=1}^{k_0} QDL^q(u_k, v_k|i_k)\; QDL^q(\hat{u}_k, \hat{v}_k |i_k).
\end{align}
where $\epsilon_k =1 $ if $\Lambda_k = (\Lambda_R)_{u_k, v_k, \hat{u}_k, \hat{v}_k}$, and $\epsilon_k=-1$ if $\Lambda_k = (\Lambda_L)_{u_k, v_k, \hat{u}_k, \hat{v}_k}$,  $\epsilon_{k_0+1} = \epsilon_1, \; i_{k_0+1} = i_1$ and
\begin{align*}
    \text{A}_n = \frac{q^{\hat{l}_1 \hat{l}_2/2 - \hat{l}_1/4}}{n^{k_0/2} \prod_{l=1}^{k_0} D^q(u_l)^{1/n} D(\hat{u}_l)^{1/n} }. 
\end{align*}
\end{proposition}
\begin{proof}
    We know that the intertwiner can be written as 
    \begin{align*}
        \Lambda_{\varphi, \tilde r}^q = \Lambda_1 \circ \dots \circ \Lambda_{k_0} \circ T_{l_1 l_2}.
    \end{align*}
    Observe that 
    \begin{align*}
        2j^2 + i-j + \epsilon_k (j^2 +i^2 - 4ij)
    \end{align*}
    gives us the power of $q$ for both $\mathcal{L}$ and $\mathcal{R}$ by letting $\epsilon_k = -1$ for $\mathcal{ L}$ and $\epsilon_k = 1$ for $\mathcal{R}$.
    \\
    Thus, the trace of this intertwiner can be written as
    \begin{align*}
        \tr (\Lambda_{\varphi,\tilde r}^q) =& \sum_{i_0, \dots, i_{k_0} = 1}^n (\Lambda_1)_{i_0, i_1} (\Lambda_2)_{i_1 i_2} \dots (\Lambda_{k_0})_{i_{k_0-1} i_{k_0}} (T_{l_1 l_2})_{i_{k_0} i_0}
        \\
        =&\sum_{i_0, \dots, i_{k_0}}^n
        q^{-2i_0l_2-i_0l_1 + \frac{i_0-i_{k_0}}{2}}
        \prod_{k=1}^{k_0}q^{i_{k}^2+ \frac{\epsilon_k}{2} ((i_k-i_{k-1})^2-2i_k i_{k-1})}
        \\
        &\times \frac{ QDL^q(u_k, v_k|i_k)QDL^q(\hat{u}_k, \hat{v}_k|i_k)}{\sqrt{n}D^q(u_k)^{1/n}D(\hat{u}_k)^{1/n}} \delta_{i_{k_0},i_0 + l_1}.
    \end{align*}
    \begin{align*}
        \tr(\Lambda_{\varphi, \tilde r}^{q}) =&A'_n \sum_{i_0, \dots, i_{k_0}}^n q^{-2(i_{k_0}-l_1)l_2-(i_{k_0}-l_1+\frac{1}{2})l_1 + i_1^2 + \frac{\epsilon_1}{2}((i_1+l_1-i_{k_0})^2-2i_1(i_{k_0}-l_1))}
        \\
        &\times q^{\sum_{k=2}^{k_0} (i_{k}^2+ \frac{\epsilon_k}{2} (i_k^2+i_{k-1}^2-4i_k i_{k-1}))} \prod_{k=1}^{k_0} QDL^q(u_k, v_k|i_k)QDL^q(\hat{u}_k, \hat{v}_k|i_k)
        \\
        =&A_n \sum_{i_1, \dots, i_{k_0}} \exp{\left(\frac{2\pi i}{n} \left(-i_{k_0} \hat{l}_2 - i_1 \hat{l}_1 + \sum_{k=1}^{k_0} i_{l+1}^2 + \frac{\epsilon_{k+1}}{2}\left( i_l^2 + i_{l+1}^2 -4i_l i_{l+1} \right) \right) \right)} \notag   
        \\
        &\times \prod_{k=1}^{k_0} QDL^q(u_k, v_k|i_k)\; QDL^q(\hat{u}_k, \hat{v}_k|i_k).
    \end{align*}
\end{proof}

\section{Quantum Dilogarithms}
Now that we have the explicit algebraic formula for the intertwiner, we try to write the discrete dilogarithm function as a restriction of a continuous one.
We use \cite{CF} to recall definitions of different dilogarithm functions.
\subsection{Fadeev quantum dilogarithm}
The small continuous quantum dilogarithm for $\hbar >0$ is the function:
\begin{align}
    li_2^{\hbar}(z) = 2\pi i \hbar \int_{\Omega}\frac{e^{(2z-\pi)t}}{4 t sinh(\pi t) sin(\pi \hbar t)} dt
\end{align}
where the domain $\Omega$ is the real line with a small interval around 0 replaced by a semi-circle above the real line in the positive imaginary direction. The integral converges for $-\frac{\pi \hbar}{2}< \mathfrak{Re}\; z < \pi + \frac{\pi \hbar}{2}$. The big continuous quantum dilogarithm is defined as 
\begin{align*}
    Li_2^{\hbar} (z) = e^{\frac{1}{2\pi i \hbar} li_2^{\hbar}(z)}
\end{align*}
satisfying 
\begin{align}
    Li_2^{\hbar}(z + \pi \hbar) &= (1 - e^{2iz + \pi i \hbar})^{-1} Li_2^{\hbar}(z) \label{2 pi shift}
    \\
    Li_2^{\hbar}(z + \pi) &= (1 + e^{\frac{2 i z}{\hbar}})^{-1} Li_2^{\hbar}(z). \label{pi shift}
\end{align}
The classical dilogarithm function is defined by 
\begin{align*}
    li_2(z) = -\int_0^z \frac{\log(1 - t)}{t}dt
\end{align*}
for the principal branch of the logarithm. We can write the small continuous quantum dilogarithm function in terms of the classical dilogarithm function as follows:
\begin{proposition}
    For every z with $0< \mathfrak{Re} \; z < \pi$
    \begin{align}
        li_2^{\hbar}(z) = li_2(e^{2iz}) + O(\hbar^2)
    \end{align}
    as $\hbar \rightarrow 0$, and this uniformly on compact subsets of the strip $\{ z \in \CC; 0< \mathfrak Re \; z < \pi \}$. \qed
\end{proposition}
We use equation (\ref{2 pi shift}) multiple times to get the following lemma
\begin{lemma}
    If $q = e^{\frac{2\pi i}{n}}$, and $u = qe^{\frac{1}{n}U}$, then for every integer $j \geq 0$
    \begin{align}
        \prod_{k=1}^j (1 + uq^{-2k}) = \frac{Li_2^{\frac{2}{n}}(\frac{\pi}{2} - \frac{\pi }{n} + \f{U}{2ni}  - \frac{2\pi j}{n}) }{Li_2^{\frac{2}{n}}(\frac{\pi}{2} - \frac{\pi}{n} + \f{U}{2ni} )}.
    \end{align}
    \qed
\end{lemma}
Recall that $QDL^q(u,v|j) =  v^{-j}\prod_{k=1}^j (1 + uq^{-2k})$. The function $li_2^{\frac{2}{n}}(z)$ is only defined in the interval $-\frac{\pi}{n} < \mathfrak Re \; z < \pi + \frac{\pi}{n}$. Thus we need 
$$-\frac{\pi}{n}<\frac{\pi}{2} - \frac{\pi }{n} + \f{U}{2ni}  - \frac{2\pi j}{n} <  \pi + \frac{\pi }{n}\;.$$
This implies
$$  -\frac{\pi}{2}< \frac{2\pi j}{n}< \frac{\pi}{2}\;.$$ 
For the interval $\left(\frac{\pi}{2}, \frac{3\pi}{2} \right)$, we use (\ref{pi shift}). 
We can always choose the interval $\left(-\frac{\pi}{2}, \frac{3\pi}{2} \right)$ for $\frac{2\pi j}{n}$ because $QDL^q(u,v|j)$ and $QDL^q(\hat{u},\hat{v}|j)$ are n periodic, which means $j$ and $j+n$ will correspond to the same $QDL^q$. In other words, whenever $2\pi >\frac{2\pi j}{n}>\frac{3\pi}{2}$, we replace $j$ with $j-n$, so that $0>\frac{2\pi (j-n)}{n}>-\frac{\pi}{2}$ 

We define 
\begin{align*}
    &I^1 = \left(-\frac{\pi}{2}, \frac{\pi}{2}\right) &&
    I^2 = \left(\frac{\pi}{2}, \frac{3\pi}{2}\right)
\end{align*}
and an analytic function in the domain $(-\pi/2, 3\pi/2)$,
\begin{align*}
QDL^q(u, v|\alpha) QD&L^q(\hat{u}, \hat{v}| \alpha) =
\\
&\begin{cases}
    e^{-\frac{\alpha (V + \tilde V)}{2\pi}} \frac{Li_2^{\frac{2}{n}}(\frac{\pi}{2} - \frac{\pi }{n} + \frac{U}{2ni}   - \alpha) Li_2^{\frac{2}{n}}(\frac{\pi}{2} - \frac{\pi }{n} + \frac{\hat{U}}{2ni}   - \alpha) }{Li_2^{\frac{2}{n}}(\frac{\pi}{2} - \frac{\pi}{n} + \frac{U}{2ni}  ) Li_2^{\frac{2}{n}}(\frac{\pi}{2} - \frac{\pi}{n} + \frac{\hat{U}}{2ni}  )} &\text{ if } \alpha \in I^1
    \\
    (1-i^n e^{\frac{U}{2}})(1 - i^n e^{\frac{\hat{U}}{2}})
    e^{-\frac{\alpha (V + \tilde V)}{2\pi}} 
    \frac{Li_2^{\frac{2}{n}}(\frac{\pi}{2} - \frac{\pi }{n} + \f{U}{2ni}  - \alpha) Li_2^{\frac{2}{n}}(\frac{\pi}{2} - \frac{\pi }{n} + \f{\hat{U}}{2ni}  - \alpha)  }{Li_2^{\frac{2}{n}}(\frac{\pi}{2} - \frac{\pi}{n} + \f{U}{2ni} ) Li_2^{\frac{2}{n}}(\frac{\pi}{2} - \frac{\pi}{n} + \f{\hat{U}}{2ni} )}  &\text{ if } \alpha \in I^2.
\end{cases}
\end{align*}

\section{Analytic formula for the invariant}
This section provides an analytic formula for the invariant so that techniques like the Poisson summation formula and saddle point method can be used later.
\\
The term in the denominator $D^q(u)^{\frac{1}{n}}$ is a ratio of hyperbolic sin or cos functions (based on n mod 4) and has a finite limit as $n \rightarrow \infty$ \cite{BWY2}.
\begin{proposition}
    Let $U, \hat{U} \in \CC$ with $e^U \neq 1 \neq e^{\hat{U}}$ and for every odd n, $q = e^{\frac{2\pi i}{n}}, u = e^{\frac{1}{n}U}, \hat{u} = e^{\frac{1}{n}\hat{U}}$, then
    \begin{align*}
        \lim_{\substack{n \rightarrow \infty \\ n = 1\; \text{mod}\; 4}} |D^q(u) D^q(\hat{u})|^{\frac{1}{n}} 
        &= 2^{-\frac{\mathfrak Im \; A + \mathfrak Im \; \tilde A}{4\pi}} \left | \frac{\cosh \frac{A - \pi i}{4} \cosh \frac{\tilde A - \pi i}{4}}{\cosh \frac{A + \pi i}{4} \cosh \frac{\tilde A + \pi i}{4}} \right|^{\frac{1}{4}}
        \\
        \lim_{\substack{n \rightarrow \infty \\ n = 3 \text{ mod } 4}}
        |D^q(u) D^q(\hat{u})|^{\frac{1}{n}} 
        &= 2^{-\frac{\mathfrak Im \; A + \mathfrak Im \; \tilde A}{4\pi}} \left | \frac{\sinh \frac{A - \pi i}{4} \sinh \frac{\tilde A - \pi i}{4}}{\sinh \frac{A + \pi i}{4} \sinh \frac{\tilde A + \pi i}{4}} \right|^{\frac{1}{4}}.
    \end{align*}
    \qed
\end{proposition}
This is an important part of the computation as we see the terms are finite and do not grow exponentially with $n$.
\\
Define linear maps
$\bold K\colon \RR^{k_0} \rightarrow \RR^{k_0}, \bold L\colon \RR^{k_0} \rightarrow \RR$  as
\begin{align}
    &K_k(\boldsymbol{\alpha}) = \frac{\epsilon_k + \epsilon_{k+1} + 2}{2}\alpha_k - \epsilon_{k+1}\alpha_{k+1} - \epsilon_k \alpha_{k-1}, \notag
    &&
    \bold L (\boldsymbol{\alpha}) = -\hat{l}_2\alpha_{k_0} - \hat{l}_1 \alpha_1.
\end{align}
We can rewrite the invariant \ref{alg formula} as 
\begin{align}
    \text{Trace } \Lambda_{\varphi,r}^q = A_n \sum_{i_1, \dots, i_{k_0} = 1}^n F_n\left( \frac{2\pi i_1}{n}, \dots, \frac{2\pi i_{k_0}}{n} \right)
\end{align}
for a map $F_n\colon\RR^{k_0} \rightarrow \CC$, where
\begin{align}
    F_n(\boldsymbol \alpha) &= \left( \prod_{k=1}^{k_0} QDL^q(u_k, v_k|\alpha_k) QDL^q(\hat{u}_k, \hat{v}_k| \alpha_k) \right) 
    \exp\left( \frac{ni}{2\pi} \bold K (\boldsymbol{\alpha}) \cdot \boldsymbol{\alpha} + i \bold L (\boldsymbol{\alpha}) \right),
\end{align}
and
\begin{align}
    A_n &= \frac{q^{\frac{1}{4}(2\hat{l}_1 \hat{l}_2 - \hat{l}_1)}}{n^{k_0/2} \prod_{l=1}^{k_0} D^q(u_l)^{1/n} D(\hat{u}_l)^{1/n} }\;. \notag
\end{align}
We will study the product of $QDL$ functions $QDL^q(u, v|\alpha)  QDL^q(\hat{u}, \hat{v} |\alpha)$ . First, recall
\begin{align*}
    Li_2^{\frac{2}{n}} 
    \left(\frac{\pi}{2} - \frac{\pi }{n} + \f{U}{2ni}  - \alpha \right) = \exp \left(\frac{n}{4\pi i} li_2^{\frac{2}{n}} \left(\frac{\pi}{2} - \frac{\pi }{n} + \f{U}{2ni}  - \alpha \right) \right) 
\end{align*}
and
\begin{align*}
    li_2^{\frac{2}{n}} \left(\frac{\pi}{2} - \frac{\pi }{n} + \f{U}{2ni}  - \alpha \right) = li_2\left( -e^{-2i\alpha} \right) + \frac{2\pi i - U}{n} \log (1 + e^{-2i\alpha}) + O \left( \frac{1}{n^2} \right).
\end{align*}
Therefore, we get
\begin{align*}
    QDL^q(u, v|\alpha)  QDL^q(\hat{u}, \hat{v} |\alpha) = \tilde g_n(\alpha) \exp \left( \frac{n}{4\pi i} \tilde f(\alpha) + O \left( \frac{1}{n}\right) \right)
\end{align*}
where 
\begin{align*}
    \tilde f(\alpha) = 2 \li (-e^{2i\alpha}) + \frac{\pi^2}{6} 
\end{align*}
and
\begin{align*}
    \tilde g_n(\alpha) = 
    \begin{cases}
        \left(\frac{1 + e^{-2i\alpha}}{2}\right)^{1-\frac{ U + \hat{U}}{4\pi i}} 
        e^{-\alpha \frac{V + \tilde V}{2\pi}} & \text{ if } \alpha \in I^1
        \\
        (1-i^n e^{\frac{U}{2}})(1 - i^n e^{\frac{\hat{U}}{2}})
        \left(\frac{1 + e^{-2i\alpha}}{2}\right)^{1-\frac{ U + \hat{U}}{4\pi i}} 
        e^{-\alpha \frac{V + \tilde V}{2\pi}}  & \text{ if } \alpha \in I^2.
    \end{cases}
\end{align*}
Thus, for $\boldsymbol{\alpha} = (\alpha_1, \dots, \alpha_{k_0})$, we can rewrite the function $F_n(\boldsymbol{\alpha})$ as
\begin{align}
    F_n(\boldsymbol{\alpha}) = g_n(\boldsymbol{\alpha}) \exp \left(\frac{n}{4\pi i} f(\boldsymbol{\alpha}) + O \left( \frac{1}{n}\right)\right) 
\end{align}
where 
\begin{align} \label{g_n}
    g_n(\boldsymbol{\alpha}) = e^{i \bold L (\boldsymbol{\alpha})} \prod_{k=1}^{k_0} \tilde g_n(\alpha_k)
\end{align}
and
\begin{align} \label{f}
    f(\boldsymbol{\alpha}) = \sum_{k=1}^{k_0} 2\li(-e^{-2i\alpha_k}) + \frac{k_0 \pi^2}{6} - 2\bold K(\boldsymbol{\alpha}) \cdot \boldsymbol{\alpha}.
\end{align}

\section{Poisson Summation Formula}
The main result of this section is Proposition \ref{poisson sum}. We change the sum from finite to infinite using a smooth bump function and then use the Poisson summation formula. 
\begin{lemma} \label{around pi/2}
    Let $U, \hat{U}, V, \tilde V \in \CC$ with $e^U = u^n, e^{\hat{U}} = \hat{u}^n, e^V = 1 + e^U, e^{\tilde V} = 1 + e^{\hat{U}}$ be given. Then for every $\delta> 0$ sufficiently small and for n sufficiently large,
    \begin{align*}
        |QDL^q(u, v|\alpha)  QDL^q(\hat{u}, \hat{v} |\alpha)| = \kappa e^{n O_1(\delta \log \delta)}
    \end{align*}
    where $\kappa>0$ is a constant that does not grow exponentially in $n$,  
    for every $\alpha \in (-\frac{\pi}{2}, \frac{3\pi}{2})$ that is at $\delta$ distance from the bridge points $-\frac{\pi}{2},  \frac{\pi}{2}, \frac{3\pi}{2}$.
\end{lemma}
\begin{proof}
    We use the fact that $ \mathfrak Im\; \text{li}_2(e^{2i\delta}) = 2 \Lambda (\delta)$ and $\Lambda(\delta) = O(\delta \log \delta)$.  Here, $\Lambda(\theta)$ is the Lobachevsky function defined by
    \begin{align*}
        \Lambda(\theta) = -\int_{0}^{\theta} \log |2\sin t| dt
    \end{align*}
    for $\theta \in [0,\pi]$.
    Observe that $\Lambda(0) = \Lambda(\pi/2) = 0$ and $\Lambda(\pi + \theta) = \Lambda(\theta)$.
    \\
    If $\alpha = \frac{\pi}{2} \pm \delta$,
    \begin{align*}
        |QDL^q(u, v|\alpha)  QDL^q(\hat{u}, \hat{v} |\alpha)| = \left|\left(\frac{1 -e^{\pm2i\delta}}{2}\right)^{1- \frac{U + \hat{U}}{4\pi i}} 
        e^{\frac{V + \tilde V}{2\pi} (\pm\delta - \frac{\pi}{2})} \right| \exp \left( \frac{n}{2\pi } O_1(\delta \log \delta) \right). 
    \end{align*}
    The term in front of the exponential is a constant that does not grow exponentially in $n$.
    Similarly, when $\alpha = -\frac{\pi}{2} + \delta$ we get the same result. For the other case, that is, $\alpha = \frac{3\pi }{2} - \delta$, we use equation (\ref{pi shift}) and it only changes the constant term.
\end{proof}
Thus, near the bridge points $-\tfrac{\pi}{2}, \tfrac{\pi}{2}, \tfrac{3\pi}{2}$, the estimate is small. 
\\
We extend the domain of the function $f(\boldsymbol{\alpha})$ from $\RR$ to $\CC$. The function $f(\boldsymbol{\alpha})$ is well-defined and holomorphic in the domain $\{ \alpha_k \in \CC \big| \mathfrak Re \; (\alpha_k) \in (-\tfrac{\pi}{2}, \tfrac{3\pi}{2}) \} $. This follows straight from the domain of the function $\text{li}_2^{\frac{2}{n}}(z) $ and our definition of $g_n$ (\ref{g_n}). 
\\
Let $\alpha_k = x_k + iy_k$ for all $k = 1, \dots, k_0$. Define
    \begin{align*}
        h(\bold{y}) = \mathfrak Im \; f(\boldsymbol{\alpha}) = \sum_{k=1}^{k_0}2\; \mathfrak Im \; \text{li}_2(-e^{-2i\alpha_k}) - 2 \; \mathfrak Im \bold K(\boldsymbol{\alpha}) \cdot \boldsymbol{\alpha}
    \end{align*}
    as a function of the variables $y_1, \dots, y_k$.

\begin{lemma} \label{convexity}
    If $0<x_k<\frac{\pi}{2}$ or $\pi< x_k < \frac{3\pi}{2}$, then the function $ h(\bold y)$ is strictly convex in the variables $y_k$. If $\frac{\pi}{2}<x_k< \pi$ or $-\frac{\pi}{2}< x_k < 0$, then the function is strictly concave in $y_k$.
\end{lemma}
\begin{proof}
    \begin{align*}
        \frac{\partial h(\bold y)}{\partial y_k} = -4\; \mathfrak Im \; \log \left( 1 + e^{-2ix_k + 2y_k} \right) - 2(\epsilon_k + \epsilon_{k+1} + 2)x_k +4 \epsilon_{k+1}x_k  + 4\epsilon_k x_{k-1}
    \end{align*}
    We get that
    \begin{align*}
        \frac{\partial^2 h(\bold y)}{\partial y_k^2} = \frac{4\sin (2x_k)}{\cos (2x_k) + \cosh (2y_k)}
    \end{align*}
    The denominator is positive as $\cosh(2y_k) \geq  1, \cos (2x_k) \geq -1, x_k \neq \frac{\pi}{2}, \text{ and } \cos (2x_k) \neq -1 $. 
    The second derivative is strictly positive
    when $0<x_k< \frac{\pi}{2}$ or $\pi < x_k < \frac{3\pi}{2}$. 
    \\
    The second derivative is strictly negative when 
    $-\frac{\pi}{2}< x_k< 0$ or $\frac{\pi}{2}< x_k< \pi$.
\end{proof}
\begin{corollary} 
    The function $h(\bold x) = \mathfrak Im \; f(\boldsymbol{\alpha})$ as a function of $\bold x$ is strictly convex in the region $-\frac{\pi}{2}< x_k< 0$ and $\frac{\pi}{2}< x_k< \pi$. $h'(\bold x)$ is strictly concave in the region $0<x_k< \frac{\pi}{2}$ and $\pi < x_k < \frac{3\pi}{2}$.
\end{corollary}
\begin{proof}
    This follows from the holomorphicity of the function $f(\boldsymbol{\alpha})$.
\end{proof}
We now see that the function $F_n$ has an exponential growth rate smaller than the volume of the mapping torus in certain regions (next Proposition). This will make our estimates easier. Note that here we use the assumption that the volume is greater than $2(k_0 -1)v_3$ where $v_3$ is the volume of an ideal hyperbolic regular tetrahedron.
\begin{proposition} \label{bridge_estimate}
    For $\epsilon>0$, we can choose a sufficiently small $\delta>0$, such that if any $\alpha_k$ is not in $\left( \delta, \frac{\pi}{2} - \delta \right) \cup \left(\pi + \delta, \frac{3\pi}{2} - \delta \right)$,
    then
    \begin{align*}
        \left| F_n(\boldsymbol{\alpha}) \right| < O \left( e^{\frac{n}{4\pi} \left( 2(k_0 - 1)v_3 + \epsilon \right) }  \right),
    \end{align*}
    where $v_3$ is the volume of an ideal regular hyperbolic tetrahedron.
\end{proposition}
\begin{proof}
    Suppose $\alpha_k$ for $k = 1, \dots, k_0$ is not in the desired domain.
    We split the proof into three cases. 

    Case I: $\alpha_k$ is near the bridge points.
    When $\alpha_k$ is $\delta$-close to $ \{-\frac{\pi}{2}, \frac{\pi}{2}, \frac{3\pi}{2} \}$, then, by lemma \ref{around pi/2}, we have the desired result.
    
    Case II:
    When $\alpha_k \in \{ \delta,  \pi + \delta\}$ for any $k \in \{1, \dots, k_0\}$,  then we have 
    \begin{align*}
        |QDL^q(u_k, v_k|\alpha_k)  QDL^q(\hat{u}_k, \hat{v}_k |\alpha_k)| = |\tilde g_n(\alpha_k)| \exp \left(\frac{n}{2\pi } \mathfrak Im \; li_2(-e^{-2i\delta}) \right) 
        \\
        = \exp \left(\frac{n}{2\pi} O_1(\delta \log \delta)  + O_2 (1)\right)
    \end{align*}
    where $O_2(1)$ comes from the fact that $|\tilde g_n(\alpha_k)|$ is bounded by some constant and $O(\delta \log \delta)$ comes from the fact that $\mathfrak Im \; \li(e^{2i(\pi/2 - \delta)}) = 2\Lambda(\pi/2 - \delta) = O(\delta \log \delta)$ .
    
    Case III: When $\alpha_k \in (-\frac{\pi}{2}, 0) \cup (\frac{\pi}{2}, \pi )$, we refer to the lemma \ref{convexity}. In the interval, the function $\mathfrak Im \; f(\boldsymbol{\alpha})$ is convex in the variables $x_k$.
    This implies the maximum is attained at the boundary. However, at the boundary, we have $x_k \in \left\{ -\frac{\pi}{2} + \delta, -\delta, \frac{\pi}{2} + \delta, \pi -\delta \right\} $, and by Lemma \ref{around pi/2}, we see that the value is small. 
\end{proof}
We add a bump function to define the invariant on a larger space with compact support, in order to apply the Poisson Summation Formula and change from sum to integrals. We look at the behavior of the function $F_n$ in the region $(-\pi/2, 3\pi/2)^{k_0}$.
Let $J = (j_1, \dots, j_{k_0})$ be a multi-index where each $j_k$ can be 0 or 1. 
For $\delta>0$, let
\begin{align*}
    D_{\delta,J}^{\RR} = \left\{ \bold x \in \RR^{k_0} \bigg|  
     j_k \pi + \delta < x_k< j_k \pi + \frac{\pi}{2}  - \delta , \forall k = 1, \dots, k_0 \right\}.
\end{align*}
Consider $ \mathcal{D}^{\RR}_{\delta}$ as union of all $D_{\delta, J}^{\RR}$. 
When $\delta = 0$, we call it $\mathcal{D}^{\RR} $. 
Similarly, we define
\begin{align*}
    D_{\delta, J}^{\CC} = \left \{ \bold z \in \CC^{k_0} \bigg| j_k \pi + \delta < \mathfrak Re \; (z_k) < j_k \pi + \frac{\pi}{2} -\delta, \forall k = 1, \dots, k_0 \right \}.
\end{align*}
Now, consider a smooth bump function
\begin{align*}
    \psi\colon \RR^{k_0} \rightarrow \RR
\end{align*}
defined by
\begin{align*}
    \begin{cases}
        \psi(\bold x) = 1 & \text{ if } \bold x \in \overline{\mathcal{D}^{\RR}_{\delta/2}}
        \\
        0 < \psi(\bold x) < 1 & \text{ if } \bold x \in \mathcal{D}^{\RR} \backslash \overline{\mathcal{D}_{\delta/2}^{\RR}}
        \\
        \psi(\bold x) = 0 & \text{ if } \bold x \notin \mathcal{D}^{\RR}.
    \end{cases}
\end{align*}
We define a function 
\begin{align*}
    G_n(\boldsymbol{\alpha}) = \psi(\boldsymbol{\alpha}) F_n(\boldsymbol{\alpha}) .
\end{align*}
We can then write our invariant as
\begin{align*}
    \tr \; \Lambda_{\varphi, r}^{q} =  A_n \sum_{(i_1, \dots, i_{k_0}) \in \ZZ^{k_0}} G_n \left(\frac{2\pi i_1}{n}, \dots, \frac{2\pi i_{k_0}}{n} \right) + O \left( e^{\frac{n}{4\pi}(2(k_0 -1)v_3 + \epsilon  }\right).
\end{align*}
Because the function $G_n$ is smooth and with compact support, we have the following classical formula, see \cite{O16} 
\begin{proposition}(Poisson Summation Formula) \label{poisson sum}
    For the function $G_n\colon\RR^{k_0} \rightarrow \CC$ defined above, 
    \begin{align*}
        \sum_{(i_1, \dots, i_{k_0})  \in \ZZ^{k_0}} G_n\left(\frac{2\pi i_1}{n}, \dots,  \frac{2\pi i_{k_0}}{n}\right) = \sum_{ \bold m \in \ZZ^{k_0}}\hat{G}_n(\bold m)
    \end{align*}
    for the Fourier coefficients
    \begin{align*}
        \hat{G}_n(\bold m) = \left(\frac{n}{2\pi}\right)^{k_0} \int_{\mathcal{D}^{\RR}} G_n(\boldsymbol\alpha) e^{\sqrt{-1}n \bold m \cdot \boldsymbol{\alpha} } d\boldsymbol{\alpha}.
    \end{align*}
    \qed
\end{proposition}

\section{Geometry of critical points}
In this section, we will see how the critical points relate to the geometry of the mapping torus $M_{\varphi}$. The main result of this section is Proposition \ref{volume}.
\subsection{The function $f$}
Recall that the function $f(\boldsymbol{\alpha})$ in equation (\ref{f}) is defined by
\begin{align*}
    f(\boldsymbol{\alpha}) = \sum_{k=1}^{k_0} 2 \text{li}_2 (-e^{-2i\alpha_k}) - 2\bold K (\boldsymbol{\alpha}) \cdot \boldsymbol{\alpha}.
\end{align*}
Then, the derivative is given by
\begin{align} \label{critical point equation}
    \frac{\partial f(\boldsymbol{\alpha  }) }{\partial \alpha_k} = 4i \log (1+e^{-2i\alpha_k}) - 2(\epsilon_k + \epsilon_{k+1} + 2)\alpha_k + 2 \epsilon_k \alpha_{k-1} + 2\epsilon_{k+1} \alpha_{k+1}.
\end{align}
The second derivatives are given by
\begin{align*}
    \frac{\partial^2 f(\boldsymbol{\alpha})}{\partial \alpha_k^2} &= \frac{8e^{-2i\alpha_k}}{1 + e^{-2i\alpha_k}} - 2(\epsilon_k + \epsilon_{k+1} + 2),
    \\
    \frac{\partial^2 f(\boldsymbol{\alpha})}{\partial \alpha_k \partial \alpha_j} &= 
    \begin{cases}
        2\epsilon_k & \text{ if } j = k-1
        \\
        2\epsilon_{k+1} & \text{ if } j = k+1.
    \end{cases}
\end{align*}
From lemma \ref{convexity}, we know that for $\alpha_k = x_k + iy_k$,
\begin{align*}
        \frac{\partial \; \mathfrak Im \;f(\boldsymbol{\alpha}) }{\partial y_k} = -4\; \mathfrak Im \; \log \left( 1 + e^{-2ix_k + 2y_k} \right) - (\epsilon_k + \epsilon_{k+1} + 2)x_k - 2\epsilon_{k+1}x_{k+1}  - 2\epsilon_k x_{k-1}.
    \end{align*}
    We get that
    \begin{align} \label{second derivate im(f)}
        \frac{\partial^2 \; \mathfrak Im \; f(\boldsymbol{\alpha})}{\partial y_k^2} = \frac{4\sin (2x_k)}{\cos (2x_k) + \cosh (2y_k)}.
    \end{align}
We call the unique critical point in the region $\{ (\alpha_1, \dots, \alpha_{k_0}) \big| (x_1, \dots, x_{k_0}) \in \left(0, \frac{\pi}{2} \right)^{k_0} \}$ as 
\begin{align} \label{critical_point}
    \boldsymbol{\alpha}^{\bold 0} = (\alpha_1^0, \dots, \alpha_{k_0}^0).
\end{align}
This critical point exists because of the existence of geometric triangulation of the four-punctured sphere bundle \cite[Appendix~A]{GF}.

\subsection{Edge weight system and shape parameters}
Recall that a pseudo-Anosov diffeomorphism $\varphi$ gives rise to an ideal triangulation sweep for $\varphi$, such that 
\begin{align*}
    \tau_k = \varphi_1 \circ \varphi_2 \circ \dots \circ \varphi_k(\tau_0).
\end{align*}
A periodic edge weight system for this ideal triangulation sweep is a family $(a_0, \dots, f_0), \dots,$ $(a_{k_0},  \dots, f_{k_0}) = (a_0,  \dots, f_0)$ where $(a_k, b_k)$ is related to $(a_{k-1}, b_{k-1})$ by (\ref{edge_weight_relation}) and the other edges are related by the central elements $p_j$ for $j=1,2,3,4$. 
Such a periodic edge weight system determines a character $[\tilde r] \in \mathcal{X}_{\PSL(\CC)}(S_{0,4})$ that is fixed under the action of $\varphi$.
\begin{proposition}
    Let $\boldsymbol{\alpha}^c$ be a smooth critical point of the function $f(\boldsymbol{\alpha})$. Then, we have a unique periodic edge weight system $(a_0,  \dots, f_0), (a_1, \dots, f_1) , \dots, (a_{k_0}, \dots, f_{k_0}) = (a_0,  \dots, f_0)$, for the ideal triangulation sweep $\tau_0, \tau_1, \dots, \tau_{k_0} = \varphi (\tau_0)$ such that
    \begin{align*}
        a_k = e^{2i\alpha_k}
    \end{align*}
    for every $k = 1, 2, \dots, k_0$.
    We also get that $p_j^n = 1$ for j = 1,2,3,4.
\end{proposition}
\begin{proof}
    By \ref{critical point equation}, the critical point equation of $f$ is 
    \begin{align*}
        4i\log(1+ e^{-2i\alpha_k}) = -4\epsilon_k \alpha_{k-1} + 2(\epsilon_k + \epsilon_{k+1} + 2)\alpha_k - 4\epsilon_{k+1} \alpha_{k+1}
    \end{align*}
    which implies that
    \begin{align*}
     (1 + e^{-2i\alpha_k})^2 = (e^{2i\alpha_{k-1}})^{\epsilon_k} (e^{2i\alpha_k})^{-\frac{\epsilon_k + \epsilon_{k+1}}{2} -1} (e^{2i\alpha_{k+1}})^{\epsilon_{k+1}}
    \end{align*}
    for $k = 1, \dots, k_0$ and
    where the indices are counted modulo $k_0$.
    \\
    We define $a_k= e^{2i\alpha_k}$ using induction. Assume that the first diagonal exchange in the diffeomorphism is $\mathcal{L} $. The other case for $\mathcal{R}$ is similar. We start with 
    \begin{align*}
        &a_0 = d_0 = e^{2i\alpha_0} & b_0 = e_0 = e^{-2i\alpha_1} && c_0 = f_0 = e^{2i(\alpha_1 - \alpha_0)}.
    \end{align*}
    Then, $a_1 = b_0^{-1} = e^{2i\alpha_1}$. We also have $p_j^n=1$ for all $j = 1,2,3,4$. This shows that $(a_k, \dots, f_k)$ is uniquely determined by $\alpha_k$ for $k=1, \dots, k_0$ as the consecutive edge weights are related by (\ref{edge_weight_relation}).
    \\
    For $k \geq 2$, assume that $a_k = e^{2i\alpha_k}$ and $a_{k-1} = e^{2i\alpha_{k-1}}$. We'll prove that $a_{k+1} = e^{2i\alpha_{k+1}}$. 
    \\
    There are four cases. $(\varphi_{k-1}, \varphi_k) = \{(L, L),( L, R), (R, L), (R, R) \}$. 
    \\
    CASE (LL). $\epsilon_k = -1 = \epsilon_{k+1}$
    \\
    From the edge weight system, we get
    \begin{align*}
        a_{k+1} = b_k^{-1} = (1 + a_k^{-1})^{-2} a_{k-1}^{-1} = (1 + e^{-2i\alpha_k})^{-2} e^{-2i\alpha_{k-1}} = e^{2i\alpha_{k+1}}.
    \end{align*}
    where the last equality comes from the critical point equation
    \begin{align*}
        (1 + e^{-2i\alpha_k})^2 = e^{-2i\alpha_{k-1} -2i\alpha_{k+1}}.
    \end{align*}
    CASE (LR). $\epsilon_k = -1, \epsilon_{k+1} = 1$
    \\
    From the edge weight system, we get
    \begin{align*}
        a_{k+1} = a_k b_k = a_k (1 + a_k^{-1})^2 a_{k-1} = e^{2i\alpha_k + 2i\alpha_{k-1}} ( 1 + e^{-2i\alpha_k})^2 = e^{2i\alpha_{k+1}}.
    \end{align*}
    where the last equality comes from the critical point equation
    \begin{align*}
        (1 + e^{-2i\alpha_k})^2 = e^{-2i\alpha_{k-1}-2i\alpha_k + 2i\alpha_{k+1}}.
    \end{align*}
    CASE (RL). $\epsilon_k = 1, \epsilon_{k+1} = -1$
    \begin{align*}
        a_{k+1} = b_k^{-1} = (1 + e^{-2i\alpha_k})^{-2} e^{-2i\alpha_k + 2i\alpha_{k-1}} = e^{2i\alpha_{k+1}}
    \end{align*}
    CASE (RR). $\epsilon_k = 1 = \epsilon_{k+1}$
    \begin{align*}
        a_{k+1} = a_k b_k = (1 + e^{-2i\alpha_k})^2 e^{4i\alpha_k - 2i\alpha_{k-1}} = e^{2i\alpha_{k+1}}
    \end{align*}
    For the $\PSL(\ZZ)$ part of the diffeomorphism, we see that we have the result. If the $\ZZ_2 \times \ZZ_2$ part of $\varphi$ is non-trivial, then that corresponds to flipping edges. Since everything until this point can be expressed in terms of $\{\alpha_k\}_{k = 1, \dots, k_0}$, the permutation of the edges will still be expressed in terms of $\{ \alpha_k\}_{k = 1, \dots, k_0}$. 
    \\
    Here, we abuse the notation a little. For $k=1, \dots, k_0$, we get our edge weights. The diffeomorphism $\varphi$ involves (up to) two permutation elements as well, which add to the edge weight system. However, since the opposite edges are the same (in terms of $\alpha_k$), the edge weight system does not change in terms of $\alpha_k$. We ignore these two elements when we talk about the periodicity of the edge weight system. 

    It's important to note that one needs to consider the permutation elements in order for the character $[\tilde r_{\text{hyp}}]$ to be $\varphi$-invariant, even though its action on the edge weight system is trivial. 
 
    This concludes that every critical point $\boldsymbol{\alpha^c}$ provides an edge weight system $(a_0, \dots, f_0), \dots,$ $ (a_{k_0}, \dots, f_{k_0})$ such that $a_k = e^{2i\alpha_k}$ for every k. The index k in $\alpha_k$ is counted modulo $k_0$, which means 
    we have $(a_{k_0}, b_{k_0}) = (e^{2i\alpha_{k_0}}, e^{-2i\alpha_{k_0+1}}) = (e^{2i\alpha_{0}}, e^{-2i\alpha_1}) = (a_0, b_0)$, where the second equality comes from the fact that $\alpha_{k_0}$ and $\alpha_0$ are the same as the indices for $\alpha$ are counted modulo $k_0$. Since, at each step, we only (possibly) permute $p_j$, we see that $p_j^{(k)} = 1$ for $j=1,2,3,4$ and $k = 1, \dots, k_0$. Thus, the edge weight system is periodic and unique.

\end{proof}
We get a unique periodic edge weight system, that corresponds to a hyperbolic character $[\tilde r_{\text{hyp}}] \in \mathcal{X}_{\PSL(\CC)}(S_{0,4})$ that is $\varphi$-invariant. This particular character (and edge weight system) will give us the hyperbolic volume of the mapping torus at the complete structure.

\subsection{Critial values and volume}
The following lemma is due to Ka Ho Wong \cite{W}. 
\begin{lemma} \label{ka ho}
    Consider a function 
    \begin{align*}
        f(\boldsymbol{\alpha}) = \sum_{k=1}^{k_0} 2 li_2(-e^{-2\alpha_k}) + Q(\boldsymbol{\alpha})
    \end{align*}
    where $Q(\boldsymbol{\alpha})$ is a quadratic polynomial with real coefficients in $\boldsymbol{\alpha}$. If $\boldsymbol{\alpha}^c$ is a critical point of $f$, then
    \begin{align*}
        \mathfrak Im \; f(\boldsymbol{\alpha}^c) = \sum_{k=1}^{k_0}2 D(-e^{-2i\alpha_k})
    \end{align*}
    where $D(z) = \mathfrak Im \; li_2(z) + \arg (1 -z)\ln |z|$ is the Bloch-Wigner dilogarithm function.
\end{lemma}
\begin{proof}
    Let $\alpha_k = x_k + y_k$. Then
    \begin{align*}
        \frac{\partial \; \mathfrak Im \; \text{li}_2(-e^{-2i\alpha_k})}{\partial y_k} = \mathfrak Im \; \frac{\partial\; \text{li}_2(-e^{-2i\alpha_k})}{\partial y_k}
        = -2 \arg (1 + e^{-2i\alpha_k}).
    \end{align*}
    The imaginary part of the quadratic term is linear in $y_k$ and has no constant term because the coefficients are real. Therefore,
    \begin{align*}
        \mathfrak Im \; Q(\boldsymbol{\alpha}) = \sum_{k=1}^{k_0} y_k\frac{\partial \; Q(\boldsymbol{\alpha})}{\partial y_k} \;.
    \end{align*}
    At the critical point (because f is holomorphic),
    \begin{align*}
        \frac{\partial \; \mathfrak Im \; Q(\boldsymbol{\alpha}^c)}{\partial y_k} = - \frac{\partial \; \mathfrak Im \; 2 \;\text{li}_2(-e^{-2i\alpha_k^c})}{\partial y_k}
    \end{align*}
    for every k. Therefore, at the critical point $\boldsymbol{\alpha}^c$
    \begin{align*}
        \mathfrak Im \; Q(\boldsymbol{\alpha}^c) = -\sum_{k=1}^{k_0}y_k \frac{\partial \; 2 \; \mathfrak Im \; \text{li}_2(-e^{-2i\alpha_k^c})}{\partial y_k} = \sum_{k=1}^{k_0} 4 y_k \arg (1 + e^{-2i\alpha_k^c})
        \\
        = \sum_{k=1}^{k_0} 2 D(-e^{-2i\alpha_k^c}) - \sum_{k=1}^{k_0} 2\; \mathfrak Im \; \text{li}_2(-e^{-2i\alpha_k^c}).
    \end{align*}
\end{proof}
By definition of the volume of a character, we get the following lemma. 

\begin{lemma} \label{volume of character}
    Let $r\colon \pi_1(M_{\varphi}) \rightarrow \text{PSL}_2(\CC)$ be associated to a periodic edge weight system. Then the volume of $r$ is equal to 
    \begin{align*}
        vol (r) = 2\sum_{k=1}^{k_0} D(-a_k^{-1}).
    \end{align*}
\end{lemma}
\begin{remark}
    Note that we get a factor of 2 in front of the Bloch-Wigner dilogarithm because the exchange involves two tetrahedra. 
\end{remark}
\begin{proposition}\label{volume}
    Let $(\boldsymbol{\alpha}^c)$ be a critical point of the function $f$ and $r\colon \pi_1(M_{\varphi}) \rightarrow \text{PSL}_2(\CC)$ be its associated homomorphism. Then, 
    \begin{align*}
        \mathfrak Im \; f (\boldsymbol{\alpha}^c) = vol(r).
    \end{align*}
\end{proposition}
\begin{proof}
    This follows from Lemma \ref{ka ho} and Lemma \ref{volume of character}.
\end{proof}

\section{Saddle point method}
We wish to apply the saddle point method, for which we need to deform the integral domain appropriately. 
Here we make a strong technical assumption that $\Vol (M_{\varphi}) >2(k_0-1)v_3 $ where $v_3$ is the volume of an ideal regular hyperbolic tetrahedron. We expect that this assumption can be lifted with some other techniques.
\begin{proposition} \cite{WY20} \label{saddle point}
    Let $f(z_1, \dots, z_{k_0})$ and $g(z_1, \dots, z_{k_0})$ be two holomorphic functions defined on a region $D \subset \CC^{k_0}$.
    Let $S$ be an embedded real $n$-dimensional closed disk in $D$ and let $(c_1, \dots, c_{k_0}) \in S$ be a point such that 
    \begin{enumerate}
        \item $(c_1, \dots, c_{k_0})$ is a critical point of $f$ in $D$,
        \item $\mathfrak Re \; (f) (z_1, \dots, z_{k_0}) < \mathfrak Re \; (f) (c_1, \dots, c_{k_0}) $ for all $(z_1, \dots, z_{k_0}) \in S \backslash \{(c_1, \dots, c_{k_0}) \}$,
        \item the domain $\{ z_1, \dots, z_{k_0}) \in D | \mathfrak Re f(z_1, \dots, z_{k_0}) < \mathfrak Re f(c_1, \dots, c_{k_0})\} $ deformation retracts to $S\backslash\{(c_1, \dots, c_{k_0} )\} $,
        \item the Hessian matrix $\Hess(f)(c_1, \dots, c_{k_0})$ is non singular,
        \item $g(c_1, \dots, c_{k_0})$ is non zero.
    \end{enumerate}
    Then, for any sequence of holomorphic functions 
    \begin{align*}
        f_n(z_1, \dots, z_{k_0}) = f(z_1, \dots, z_{k_0}) + \frac{v_n(z_1, \dots, z_{k_0})}{n^2}
    \end{align*}
    where $|v_n (z_1, \dots, z_{k_0}) |$ is bounded from above by a constant independent of n in D.
    \\
    Then,
    \begin{align*}
        \int_S g(z_1, \dots, z_{k_0}) &e^{n f_n(z_1, \dots, z_{k_0})} dz_1 \dots dz_{k_0} 
        \\
        = &\left(\frac{2\pi}{n} \right)^{\frac{k_0}{2}} \frac{g(c_1, \dots, c_{k_0})}{\sqrt{(-1)^{k_0}\det \Hess (f)(c_1, \dots, c_{k_0}) }} e^{nf(c_1, \dots, c_{k_0})} \left( 1 + O\left( \frac{1}{n}\right) \right).
    \end{align*}
    \qed
\end{proposition}
\begin{remark}
    There is a typo in \cite{WY20}. The exponent of $(-1)$ is missing there.
\end{remark}
We will estimate the following:
\begin{align*}
    \int_{\mathcal{D}^{\RR}} \psi(\boldsymbol{\alpha}) g_n(\boldsymbol{\alpha})e^{\frac{n}{4\pi i} f(\boldsymbol{\alpha}) + O\left( \frac{1}{n} \right)} e^{\sqrt{-1} n \bold m \cdot \boldsymbol{\alpha}} d\boldsymbol{\alpha}.
\end{align*}
The coefficient of n in the exponent is $ \frac{f(\boldsymbol{\alpha})}{4\pi i} + O\left(\frac{1}{n^2} \right)$. Therefore, we have the sequence of holomorphic functions we want.

Let $\alpha_j = x_j + iy_j $ for all $j \in \{ 1, \dots, k_0\}$. For $\delta>0$ small enough, choose the region $\mathcal{D}^{\CC}_{\delta}$, and
let $\{\boldsymbol{\alpha}^J \in \mathcal{D}^{\CC}_{\delta, J} | J \in \{0,1\}^{k_0}\} $ be the set of critical points where for each fixed $(j_1, \dots, j_{k_0})$. The critical point is unique in their respective regions. We can assume that there is a critical point in the interval $(0,\pi)^{k_0}$. For each of these intervals, if any of $x_j>\frac{\pi}{2}$, then we get a concave up function in $x_j$, and we will get a value lesser than the value at the boundary, which is lesser than $\Vol ({M_{\varphi}})$. Therefore, the critical point will be in the interval $\left(0, \tfrac{\pi}{2} \right)^{k_0} \subset \left(0, \pi\right)^{k_0} $.
\\

Let $\boldsymbol{\alpha}^{\bold 0} = ((x_1^0 + iy_1^0), \dots , (x_{k_0}^0 + iy_{k_0}^0))$ be the unique critical point where $\bold 0 = (0, \dots, 0)$. Let $S^{\bold 0} = S^{\text{top}} \cup S^{\text{side}}$ 
where
\begin{align*}
    S^{\text{top}} &= \left\{ (x_1 + iy_1^0, \dots, x_{k_0} + iy_{k_0}^0)  \big| (x_1, \dots, x_{k_0}) \in \left(\delta,\frac{\pi}{2}-\delta \right)^{k_0} \right\}
    \\
    S^{\text{side}} &= \left\{ (x_1 + ity_1^0, \dots, x_{k_0} + ity_{k_0}^0) \big | t \in [0,1], (x_1, \dots, x_{k_0}) \in \partial \mathcal{D}^{\RR}_{\delta ,\bold 0}\right \}.
\end{align*}
In figure \ref{fig:S^0}, we can see the cube when there are only two variables $x_1$ and $x_2$.
\begin{figure}[h]
    \centering
    \includegraphics[width = 0.75\textwidth]{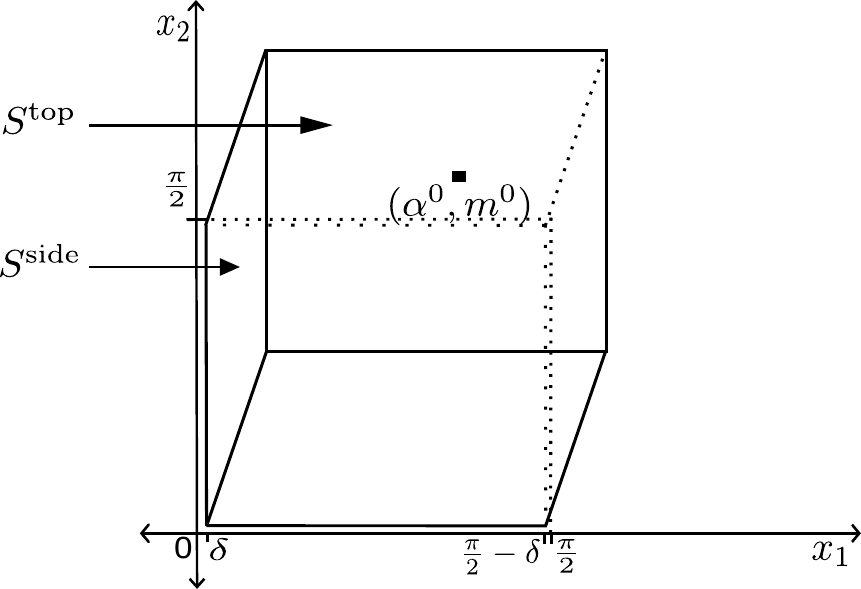}
    \caption{The domain $S^0$ simplified for the case where $k_0=2$. In general, it's a multi-dimensional cube structure. $S^{\text{top}}$ contains the critical point. Here, $S^{\text{side}}$ consists of all the four sides. }
    \label{fig:S^0}
\end{figure}

Similarly, we define $S^J= \tilde S^{J,\text{top}} \cup \tilde S^{J,\text{side}}$ for $J \neq \bold 0$, where 
\begin{align*}
    \tilde S^{J,\text{top}} &= \left\{ (x_1 + iy_1^0, \dots, x_{k_0} + iy_{k_0}^0)  \big| (x_1, \dots, x_{k_0}) \in D^{\RR}_{\delta,J} \right\}
    \\
    \tilde S^{J,\text{side}} &= \left\{ (x_1 + ity_1^0, \dots, x_{k_0} + ity_{k_0}^0) \big | t \in [0,1], (x_1, \dots, x_{k_0}, m_1, \dots, m_{k_0}) \in \partial \mathcal{D}^{\RR}_{\delta,J} \right \}.
\end{align*}
\subsection{Estimate of the leading Fourier coefficient on $D_{\delta, \textbf{0}}^{\CC}$}

We show that the $(0,\dots, 0)$-th Fourier coefficient has an exponential growth rate equal to the volume of the manifold on the region $D_{\delta, \bold 0}^{\CC}$. We do this by deforming the integral domain $D_{\delta, \bold 0}^{\RR}$ to $S^{\bold 0}$ and estimating $\mathfrak Im \; f(\boldsymbol{\alpha})$. 
\begin{proposition} \label{unique_critical}
    $\mathfrak Im \; f (\boldsymbol{\alpha})$ achieves the only absolute maximum at $\boldsymbol{\alpha}^{\bold 0}$ on $S^{\bold 0}$. 
\end{proposition}

\begin{proof}
    On $S^{\text{side}}$, for any fixed $\bold x \in \mathcal{D}^{\RR}_{\delta, \bold 0}$, $h_{\bold x}(t) = \mathfrak Im \; f(\boldsymbol{\alpha}_t)$  is concave up in t by Lemma \ref{convexity}. Therefore, it suffices to check the value of $h_{\bold x}(t)$ at the boundary, $t=0 $ and $t=1$.
    \\
    At t=0, we have $h_{\bold x}(0) = \mathfrak Im \; f (\bold x)$. Here at least one of the $x_j = \delta$ or $\frac{\pi}{2}-\delta$. By arguments from lemma \ref{bridge_estimate}, we see that $\mathfrak Im \; f(\bold x) < \delta + 2(k_0-1)v_3 <\Vol (M_{\varphi})$, where $v_3$ is the volume of a regular ideal tetrahedron.

    When $t=1$, $h_{\bold x}(1) = \mathfrak Im \; h_{\infty}(x_1 + iy_1^0,\dots, x_{k_0} + iy_{k_0}^0)$. Now, on $S^{\text{top}}$, the function $f(\boldsymbol{\alpha})$ is concave down in the $\mathfrak Re (\boldsymbol{\alpha})$ by lemma \ref{convexity}. In that region, we get a unique maximum and the value at the boundary is strictly less than the value at the top. 
    \begin{align*}
        \mathfrak Im \; f(\boldsymbol{\alpha}^{\bold 0}) = \Vol (M_{\varphi}) > \mathfrak Im \; f (\boldsymbol{\alpha})
    \end{align*}
    where $\boldsymbol{\alpha} \neq \boldsymbol{\alpha}^0$. This includes $h_{\bold x}(1)$. For all $\bold x \in \mathcal{D}^{\RR}_{\delta, \bold 0} $, we have
    \begin{align*}
        \mathfrak Im \; f(x_1 + iy_1^0,\dots, x_{k_0} + iy_{k_0}^0 ) \leq \Vol (M_{\varphi}).
    \end{align*}
    Thus, $\mathfrak Im f(\boldsymbol{\alpha})$ on $S^{\text{side}}$ is smaller than the volume, and it equals the volume at a unique point $\boldsymbol \alpha^{\bold 0}$ in $S^{\text{top}}$.
\end{proof}
\begin{remark}
    Note that we wish to study the function $f(\boldsymbol{\alpha}) - 4\pi \bold m \cdot \boldsymbol{\alpha}$ where $\bold m$ is the Fourier coefficient. The above proposition shows us that the $(0, \dots, 0)$-th Fourier coefficient gives us the volume. 
\end{remark}
We now show that the hessian is non-singular at the critical point. 
\begin{lemma} \label{hessian}
    The Hessian of the function $f(\boldsymbol{\alpha})$ is non-singular at $\boldsymbol{\alpha}^c$, where $\boldsymbol{\alpha}^c$ is a critical point of $f$.
\end{lemma}
\begin{proof}
    We first show that the Hessian of the imaginary part of $f, \Hess (\mathfrak Im \; f(\boldsymbol{\alpha}))$ is positive definite at the critical point. 
    \\
    Using equation (\ref{second derivate im(f)}), we see that the diagonal terms in the Hessian are $\frac{4 \sin(2x_k)}{\cos(2x_k) + \cosh(2y_k)}$. The off-diagonal terms are 0, that is, $\tfrac{\partial^2 \mathfrak Im \; f(\boldsymbol{\alpha}) }{\partial y_k \partial y_j} = 0$. 
    The critical points lie in the interval $x_k \in (0,\tfrac{\pi}{2}) \cup (\pi, 3\tfrac{\pi}{2}) $. In this interval, all the diagonal terms are positive. Therefore, the Hessian is positive definite at the critical point.
    \\
    Using \cite[Lemma~7.2]{AGN}, if $\Hess(\mathfrak Im f(\boldsymbol{\alpha}) )$ is positive definite at the critical point and $\Hess(\mathfrak Re f(\boldsymbol{\alpha}))$ is symmetric (by definition of hessian), then $\Hess(f(\boldsymbol{\alpha}))$ is non-singular at the critical point. 
\end{proof}
\begin{proposition}
    We have the following estimate
    \begin{align*}
        \left|\int_{D_{\delta, \bold 0}^{\CC}} \psi(x_1, \dots, x_{k_0}) g_n(\alpha_1, \dots, \alpha_{k_0}) e^{\frac{n}{4\pi i} f(\alpha_1, \dots, \alpha_{k_0}) + O \left(\frac{1}{n} \right)} d\alpha_1 \dots d\alpha_{k_0} \right|
        \\
        = \frac{c'}{\sqrt{ (-1)^{k_0}\det \Hess (f)(\alpha_1^0, \dots, \alpha_{k_0}^0 )}} e^{\frac{n}{4\pi} \Vol (M_{\varphi})}  \left( 1 + O \left(\frac{1}{n} \right) \right) 
    \end{align*}
    where $c'>0$ is a constant.
\end{proposition}
\begin{proof}
    By analyticity, the integral remains the same if we deform the domain $D_{\delta, \bold 0}^{\CC}$ to $S^{\bold 0}$. We apply proposition \ref{saddle point} with region $D_{\delta, \bold 0}^{\CC}$, embedded disk $S^{\bold 0}$, the functions $f(\alpha_1, \dots, \alpha_{k_0})$, 
    \\
    $\psi(x_1, \dots, x_{k_0})g_n(\alpha_1, \dots, \alpha_{k_0})$ and the critical point $(\alpha_1^0, \dots, \alpha_{k_0}^0)$ given the conditions are satisfied.
    \\

    (1) and (2) of Proposition \ref{saddle point} are satisfied by Proposition \ref{unique_critical}. 
    \\
    
    For (3), we need to show that the domain 
    \begin{align*}
        \left\{ (\alpha_1, \dots, \alpha_{k_0}) \in D_{\delta, \bold 0}^{\CC}\; |\; \mathfrak Im \; f(\alpha_1, \dots, \alpha_{k_0}) < \mathfrak Im \; f(\alpha_1^0, \dots, \alpha_{k_0}^0) \right\}
    \end{align*}
    deformation retracts to $S \backslash (\alpha_1^0, \dots, \alpha_{k_0}^0)$. 
    Let $(\theta_1, \dots, \theta_{k_0}) \in D^{\RR}_{\delta, 0}$. Define 
    \begin{align*}
        P_{\theta_1, \dots, \theta_{k_0}} = \{ (\alpha_1, \dots, \alpha_{k_0}) \in D_{\delta, \bold 0}^{\CC} \; |\; \mathfrak Re \; (\alpha_k) = \theta_k \text{ for } k = 1, \dots, k_0 \} 
    \end{align*}
    and 
    \begin{align*}
        D_{\theta_1, \dots, \theta_{k_0}} = \{ (\alpha_1, \dots, \alpha_{k_0}) \in P_{\theta_1, \dots, \theta_{k_0}} \; | \; \mathfrak Im \; f(\alpha_1, \dots, \alpha_{k_0}) < \mathfrak Im \; f(\alpha_1^0, \dots, \alpha_{k_0}^0) \}.
    \end{align*}
    By Lemma \ref{convexity}, we see that $D_{\mathfrak Re \; (\alpha_1), \dots, \mathfrak Re \; (\alpha_{k_0}) } = \emptyset$ as the absolute minimum is achieved in the variable $\mathfrak Im \; (\alpha_k)$ for $k = 1, \dots, k_0$ in the region $D_{\delta, \bold 0}^{\CC}$.
    \\
    $D_{\theta_1, \dots, \theta_{k_0}}$ is homeomorphic to a disk for $(\theta_1, \dots, \theta_{k_0}) \neq (\mathfrak Re (\alpha_1), \dots, \mathfrak Re  (\alpha_{k_0}))$. We deformation retract $D_{\theta_1, \dots, \theta_{k_0}}$ to $\{\theta_1 + \sqrt{-1} \mathfrak Im  (\alpha_1^0), \dots, \theta_{k_0} + \sqrt{-1} \mathfrak Im  (\alpha_{k_0})\} $.
    \\
    
    By Lemma \ref{hessian}, the determinant of the Hessian (4) is non-zero at the critical point.
    \\
    
    For (5), we get
    \begin{align*}
        g(\alpha_1^0, \dots, \alpha_{k_0}) = e^{i \bold L(\boldsymbol{\alpha}}) \prod_{k=1}^{k_0} \tilde g_n(\alpha_k^0).
    \end{align*}    
    Observe that none of the $\tilde g_n(\alpha_k)$ is zero, because $e^{-2i\alpha_k} \neq -1$ when $\alpha_k \in (\delta, \frac{\pi}{2} - \delta) $.

    For the higher-order part of the function $f(\boldsymbol{\alpha})$, we have $O\left( \frac{1}{n^2}\right)$, which means $v_n$ is bounded above by a constant independent of n. 
\end{proof}
\subsection{Estimate of other Fourier coefficients in the region $D_{\delta, \textbf{0}}^{\CC}$}
\begin{proposition} \label{other fourier coeff}
    There is an $\epsilon>0$, such that when $m_k \neq 0$ for some $k = 1, \dots, k_0$, then
    \begin{align*}
        \left| \int_{D_{\delta, \bold 0}^{\CC}} \psi(\bold x) g_n(\boldsymbol{\alpha}) e^{\frac{n}{4\pi i} (f(\boldsymbol{\alpha}) - 4\pi \bold m \cdot \boldsymbol{\alpha}) } d\boldsymbol{\alpha} \right| \leq O \left(e^{\frac{n}{4\pi} \left(\Vol (M_{\varphi}) - \epsilon \right)}\right).
    \end{align*}
\end{proposition}
\begin{proof}
    For simplicity, assume $m_1 \neq 0$. The same proof works if any other $m_k \neq 0$.
    For $L>0$, define
    \begin{align*}
        S_{L,\text{top}}^{\pm} =\left \{ \boldsymbol{\alpha} \in D_{\delta, \bold 0}^{\CC} \; \big | \; \mathfrak Im (\alpha_k) = 0 \text{ for } k = 2, \dots, k_0, \mathfrak Im (\alpha_1) = \pm L \right \}
    \end{align*}    
    and 
    \begin{align*}
        S_{L,\text{side}} ^{\pm}= \left \{ (x_1 + \sqrt{-1}l, x_2, \dots, x_{k_0}) \; \big | \; (x_1, \dots, x_{k_0}) \in \partial D_{\delta, \bold 0}^{\RR}, \; l \in [0,L] \right\}.
    \end{align*}
    Here, $S_L^{\pm} = S_{L,\text{side}}^{\pm} \cup S_{L,\text{top}}^{\pm}$ is homotopic to $D_{\delta, \bold 0}^{\RR}$. 
    We split the proof into two cases.
    \\
    Case I: $m_1 > 0$. 
    We will show that for large enough L, there is an $\epsilon>0$ such that 
    \begin{align*}
        \mathfrak Im \;\left( f(\boldsymbol{\alpha}) - 4\pi \bold m \cdot \boldsymbol{\alpha} \right) < \Vol (M_{\varphi}) - \epsilon
    \end{align*}
    for $\bold m \neq (0, \dots, 0)$ on $S_L^+$. Then the result will follow from Proposition \ref{saddle point}. 
    \\
    We let $\alpha_1 = \mathfrak Re \; (\alpha_1) + \sqrt{-1}l $ for $l>0$. Then
    \begin{align*}
        \mathfrak Im \; (f(\boldsymbol{\alpha}) - 4\pi \bold m \cdot \boldsymbol{\alpha}) = \mathfrak Im \; f(\boldsymbol{\alpha}) -4\pi m_1 l .
    \end{align*}
    The derivative with respect to $l$ gives us
    \begin{align*}
        \frac{\partial \; (\mathfrak Im \; f(\boldsymbol{\alpha}) - 4\pi m_1 l) }{\partial l} = -4\arg (1 + e^{-2i\alpha_1}) - 2(\epsilon_1 + \epsilon_2 + 2)x_1 + 4\epsilon_2 x_2 + 4\epsilon_1 x_{k_0} - 4\pi m_1
        \\
        < 8x_1 - 2(\epsilon_1 + \epsilon_2 + 2)x_1 + 4\epsilon_2 x_2 + 4\epsilon_1 x_{k_0} - 4\pi m_1 
    \end{align*}
    where the last inequality comes from the fact that $-2x_1< \arg(1 + e^{-2i\alpha_1}) < 0$.
    \\
    We deal with this inequality on a case-by-case basis. We use the fact that $m_1>0$ and $0<x_k< \frac{\pi}{2}$ for $k 
=1, \dots, k_0$. Let $\delta = \min \{ |x_k - \frac{\pi}{2}|; k = 1, \dots, k_0 \}> 0$. Then,
    \\
    Case a: If $\epsilon_1 = -1 = \epsilon_2$, then
    \begin{align*}
         \text{RHS}< 8x_1 - 4x_2 - 4x_{k_0} - 4\pi m_1<-8\delta.
    \end{align*}
    Case b: If $\epsilon_1 = -1, \epsilon_2 = 1$, then
    \begin{align*}
         \text{RHS}< 4x_1 + 4 x_2 - 4x_{k_0} - 4\pi m_1<-8\delta.
    \end{align*}
    Case c: If $\epsilon_1 = 1, \epsilon_2 = -1$, then
    \begin{align*}
         \text{RHS} < 4x_1  - 4x_2 + 4x_{k_0} - 4\pi m_1< -8\delta.
    \end{align*}
    Case d: If $\epsilon_1 = 1 = \epsilon_2$, then
    \begin{align*}
         \text{RHS}< 4x_2 + 4x_{k_0} - 4\pi m_1 < -8\delta.
    \end{align*}
    Thus, the function is strictly decreasing in $l$. We push the domain along the $\sqrt{-1} l$ direction far enough, the imaginary part of $f(\boldsymbol{\alpha}) - 4\pi m_1 \alpha_1$ becomes as small as possible. Therefore, there is $\epsilon>0$ such that
    \begin{align*}
        \mathfrak Im \; \left(f(\boldsymbol{\alpha}) - 4\pi \bold m \cdot \boldsymbol{\alpha} \right)< \Vol (M_{\varphi}) - \epsilon
    \end{align*}
    on $S_{L,\text{top}}^+$. 
    \\
    $\mathfrak Im \; (f(\boldsymbol{\alpha}) - 4\pi \bold m \cdot \boldsymbol{\alpha})$ is smaller than the volume of $M_{\varphi}$ on $\partial S_{L,\text{top}}$ by lemma \ref{convexity}. Thus, it becomes even smaller on the side. 
    \begin{align*}
        \mathfrak Im \; (f(\boldsymbol{\alpha}) - 4\pi \bold m \cdot \boldsymbol{\alpha}) < \Vol (M_{\varphi}) - \epsilon
    \end{align*}
    on $S_{L,\text{side}}^+$.

    The case when $m_1<0$ follows similarly. We work with $S_L^-$. Let $\alpha_1 = \mathfrak Re \; (\alpha_1) - \sqrt{-1}l$. Then
    \begin{align*}
        \mathfrak Im \; (f(\boldsymbol{\alpha}) - 4\pi \bold m \cdot \boldsymbol{\alpha}) = \mathfrak Im \; f(\boldsymbol{\alpha}) + 4\pi m_1 l 
    \end{align*}
    The derivative with respect to $l$ gives us
    \begin{align*}
        \frac{\partial \; (\mathfrak Im \; f(\boldsymbol{\alpha}) + 4\pi m_1 l) }{\partial l} 
        =& 4\arg (1 + e^{-2i\alpha_1}) + 2(\epsilon_1 + \epsilon_2 + 2)x_1 - 4\epsilon_2 x_2 - 4\epsilon_1 x_{k_0} + 4\pi m_1
        \\
        <& 2(\epsilon_1 + \epsilon_2 + 2)x_1 - 4\epsilon_2 x_2 - 4\epsilon_1 x_{k_0} + 4\pi m_1 
    \end{align*}
    where the last inequality comes from the fact that $-2x_1< \arg(1 + e^{-2i\alpha_1}) < 0$.
    \\
    We deal with this on a case-by-case basis. We use the fact that $m_1<0$ and $0<x_k< \frac{\pi}{2}$. Let $\delta = \min \{ |x_k - \frac{\pi}{2}|; k = 1, \dots, k_0 \}> 0$. Then,
    \\
    Case I: If $\epsilon_1 = -1 = \epsilon_2$, then 
    \begin{align*}
         \text{RHS}< 4x_2 + 4x_{k_0} + 4\pi m_1< -8\delta.
    \end{align*}
    Case II: If $\epsilon_1 = -1, \epsilon_2 = 1$, then
    \begin{align*}
         \text{RHS}< 4x_1 - 4 x_2 + 4x_{k_0} + 4\pi m_1<-8\delta.
    \end{align*}
    Case III: If $\epsilon_1 = 1, \epsilon_2 = -1$, then
    \begin{align*}
         \text{RHS} < 4x_1  + 4x_2 - 4x_{k_0} + 4\pi m_1< -8\delta.
    \end{align*}
    Case IV: If $\epsilon_1 = 1 = \epsilon_2$, then
    \begin{align*}
         \text{RHS}< 8x_1 - 4x_2 - 4x_{k_0} + 4\pi m_1 < -8\delta.
    \end{align*}
    We push the domain along the $-\sqrt{-1} l$ direction far enough, that the imaginary part of $f(\boldsymbol{\alpha}) - 4\pi m_1 \alpha_1$ becomes as small as possible. Therefore, there is $\epsilon>0$ such that
    \begin{align*}
        \mathfrak Im \; \left(f(\boldsymbol{\alpha}) - 4\pi \bold m \cdot \boldsymbol{\alpha} \right)< \Vol (M_{\varphi}) - \epsilon
    \end{align*}
    on $S_{L,\text{top}}^-$. 
    \\
    $\mathfrak Im \; (f(\boldsymbol{\alpha}) - 4\pi \bold m \cdot \boldsymbol{\alpha})$ is smaller than the volume of $M_{\varphi}$ on $\partial S_{L,\text{top}}$ by lemma \ref{convexity}. Thus, it becomes even smaller on the side. 
    \begin{align*}
        \mathfrak Im \; (f(\boldsymbol{\alpha}) - 4\pi \bold m \cdot \boldsymbol{\alpha}) < \Vol (M_{\varphi}) - \epsilon
    \end{align*}
    on $S_{L,\text{side}}^-$.
    
    Therefore, on both $S_L^{\pm}$, the imaginary part of $f(\boldsymbol{\alpha}) - 4\pi \bold m\cdot \boldsymbol{\alpha}$ is less than the volume. 
\end{proof}

\subsection{Estimate of the leading Fourier coefficients in $D_{\delta, J}^{\CC}$}
We find the leading Fourier coefficients for each $D_{\delta, J}^{\CC}$ and show that it has a growth equal to the volume of the manifold. 
\begin{lemma} \label{J Fourier coeff}
    For each $J = (j_1, \dots, j_{k_0}) $, the leading Fourier coefficient in the domain $D_{\delta, J}^{\CC}$ is given by
    \begin{align*}
        \bold m = -\bold K (J).
    \end{align*}
\end{lemma}
\begin{proof}
    Consider the function 
    \begin{align*}
        \mathfrak Im \; f(\boldsymbol{\alpha} + \pi J) 
        = \mathfrak Im \; f(\boldsymbol{\alpha}) - \mathfrak Im \;4\pi \bold K(J) \cdot \boldsymbol{\alpha} .
    \end{align*}
    Then,
    \begin{align*}
         \mathfrak Im \;( f(\boldsymbol{\alpha} + \pi J) + 4\pi \bold K(J) \cdot \boldsymbol{\alpha}) = \mathfrak Im \; f(\boldsymbol{\alpha}).
    \end{align*}
    Observe that,
    \begin{align} \label{J-th critical point}
        \frac{\partial \; \mathfrak Im \;( f(\boldsymbol{\alpha} + \pi J) + 4\pi \bold K(J) \cdot \boldsymbol{\alpha})}{\partial \alpha_k} = \frac{\partial \; \mathfrak Im \; f(\boldsymbol{\alpha})}{\partial \alpha_k} = 0.
    \end{align}
    Therefore, if $\boldsymbol{\alpha}^J = (\boldsymbol{\alpha}^{\bold 0} + \pi J)$ is a critical point, then $\mathfrak Im \; (f(\boldsymbol{\alpha}^J) + 4\pi \bold K(J) \cdot \boldsymbol{\alpha}^J) = \mathfrak Im \; f(\boldsymbol \alpha^{\bold 0}) = \Vol (M_{\varphi})$. 
\end{proof}
\begin{remark}
    Notice that the critical points are just translations by $\pi$. In total, there are $2^{k_0}$ critical points.
\end{remark}
We deform the integral domain $D_{\delta, J}^{\CC}$ to $\tilde S^{J}$ and estimate $\mathfrak Im \; (f(\boldsymbol{\alpha}) + 4\pi K(J) \cdot \boldsymbol{\alpha})$.
\begin{proposition}
    The function
    \begin{align*}
        \mathfrak Im \; (f(\boldsymbol{\alpha}) + 4\pi K(J) \cdot \boldsymbol{\alpha})
    \end{align*}
    achieves the only absolute maximum at $\boldsymbol{\alpha}^J$ on $S^J$.
\end{proposition}
\begin{proof}
    Follows directly from \ref{J-th critical point}, Proposition \ref{unique_critical} along with Lemma \ref{J Fourier coeff}. 
\end{proof}
\begin{proposition}
    For each $J \in \{0,1\}^{k_0}$, we have the following estimate
    \begin{align*}
        \left|\int_{D_{\delta, J}^{\CC}} \psi(\alpha_1, \dots, \alpha_{k_0}) g_n(\alpha_1, \dots, \alpha_{k_0}) e^{\frac{n}{4\pi i} \left(f(\alpha_1, \dots, \alpha_{k_0}) + 4\pi \bold K(J) \cdot \boldsymbol{\alpha} \right) + O \left(\frac{1}{n} \right)} d\alpha_1 \dots d\alpha_{k_0} \right|
        \\
        = \frac{c_J}{\sqrt{ (-1)^{k_0}\det \Hess (f)(\boldsymbol{\alpha}^J)}} e^{\frac{n}{4\pi} \Vol (M_{\varphi})} \left( 1 + O \left( \frac{1}{n}\right) \right)
    \end{align*}
    where $c_J$ is a constant not equal to 0. 
\end{proposition}
\begin{proof}
    Since the domain is shifted by $\pi$, conditions (1), (2), (3), and (6) are verified. 
    \\
    For (4), notice that the terms in the Hessian that depend on $\alpha_k$ are of the form $e^{-2i\alpha_k}$. Adding a $\pi$ to it does not change the value, therefore, the determinant of the hessian remains unchanged.
    \\
    For (5), when we have $\alpha_k + \pi$ instead of $\alpha_k$, we see a factor of $(1 - i^n e^{\frac{U_k}{2}})(1 - i^n e^{\frac{\hat{U}_k}{2}}) \neq 0$ because $e^{\frac{U_k}{2}} \neq \pm i$ as $e^{U_k} \neq -1$ by assumption.
\end{proof}

\subsection{Estimate of other Fourier coefficients in the region $D_{\delta, J}^{\CC}$}
\begin{proposition}
    There is an $\epsilon>0$, such that when $\bold m \neq -\bold K(J) $, then
    \begin{align*}
        \left| \int_{D_{\delta, J}^{\CC}} \psi(\boldsymbol{\alpha}) g(\boldsymbol{\alpha}) e^{\frac{n}{4\pi i} (f(\boldsymbol{\alpha}) - 4\pi \bold m \cdot \boldsymbol{\alpha}) } d\boldsymbol{\alpha} \right| \leq O \left(e^{\frac{n}{4\pi} \left(\Vol (M_{\varphi}) - \epsilon \right)}\right).
    \end{align*}
\end{proposition}
\begin{proof}
    For simplicity, assume $m_1 \neq -(\epsilon_1 + \epsilon_2 + 2)j_1/2 + 2\epsilon_1 j_{k_0} + 2\epsilon_2 j_2$. The two cases are: 
    \\
    Case I: $m_1 >-(\epsilon_1 + \epsilon_2 + 2)j_1/2 + 2\epsilon_1 j_{k_0} + 2\epsilon_2 j_2 $
    \\
    Case II: $m_1 < -(\epsilon_1 + \epsilon_2 + 2)j_1/2 + 2\epsilon_1 j_{k_0} + 2\epsilon_2 j_2$.
    \\
    Observe that 
    \begin{align*}
        \mathfrak Im \; (f(\boldsymbol{\alpha} + \pi J) - 4\pi \bold m \cdot \boldsymbol{\alpha}) 
        &= \mathfrak Im \; (f(\boldsymbol{\alpha} + \pi J) + 4\pi \bold K(J) \cdot \boldsymbol{\alpha} - 4\pi (\bold m + \bold K(J)) \cdot \boldsymbol{\alpha})
        \\
        &= \mathfrak Im\; f(\boldsymbol{\alpha}) - \mathfrak Im \; (4\pi (\bold m + \bold K(J)) \cdot \boldsymbol{\alpha}).
    \end{align*}
    Thus, by a change of variable $\boldsymbol{\alpha} \in D_{\delta, J}^{\CC}$ to $\boldsymbol{\alpha} + \pi J \in D_{\delta, J}^{\CC} $ and then shifting the domain by subtracting $\pi J$, we have the same estimate as the previous Proposition \ref{other fourier coeff}. 
\end{proof}

\section{Asymptotics}
We first show that the sum of the smaller Fourier coefficients is still small. Then, the leading Fourier coefficients determine the sum and we prove the conjecture.
\begin{proposition} \label{sum of small Fourier coeff}
    The sum of all Fourier coefficients other than the $2^{k_0}$ leading ones is small.
    \begin{align*}
        \sum_{(\boldsymbol{\alpha}, \bold m) \neq (\boldsymbol{\alpha}^J, \bold K(J))} |\hat{G}_n(\boldsymbol{\alpha}, \bold m)| \leq O(e^{\frac{n}{4\pi} (\Vol (M_{\varphi}) - \epsilon)})
    \end{align*}
\end{proposition}
\begin{proof}
    This follows straight from the fact that $\hat{G}_n$ has compact support, each term of the summand is bounded by $O(e^{\frac{n}{4\pi} (\Vol(M_{\varphi}) - \epsilon)})$ and the Fourier transformation has a convergence tail.  
    For each $\bold m$, consider the function $\psi(\boldsymbol{\alpha}) g_n(\boldsymbol{\alpha})$. For each $\bold m$, since $\psi(\boldsymbol{\alpha})$ vanishes outside of $\mathcal{D}^{\RR}$, by integration by parts we have 
    \begin{align*}
        r^{2k_0}\sum_{j=1}^{k_0} m_j^{2k_0} \int_{\mathcal{D}^{\RR}} \psi(\boldsymbol{\alpha}) &g_n(\boldsymbol{\alpha}) 
        e^{\frac{n}{4\pi \sqrt{-1}} (f(\boldsymbol \alpha) - 4\pi \bold m \cdot \boldsymbol{\alpha})} d\boldsymbol{\alpha}
        \\
        &= \int_{\mathcal{D}^{\RR}} \psi(\boldsymbol{\alpha}) g_n(\boldsymbol{\alpha}) e^{\frac{n}{4\pi \sqrt{-1}} f(\boldsymbol{\alpha})} \left( \left( \sum_{j=1}^{k_0} \frac{\partial^{2k_0}}{\partial \alpha_j^{2k_0}} \right) e^{\frac{n}{4\pi \sqrt{-1}} (-4\pi \bold m \cdot \boldsymbol{\alpha})} \right) d\boldsymbol{\alpha}
        \\
        &= \int_{\mathcal{D}^{\RR}} \left( \left(\sum_{j=1}^{k_0} \frac{\partial^{2k_0}}{\partial \alpha^{2k_0}} \right) \left(\psi(\boldsymbol{\alpha}) g_n(\boldsymbol{\alpha}) e^{\frac{n}{4\pi \sqrt{-1}} f(\boldsymbol{\alpha})}\right) \right) e^{\frac{n}{4\pi \sqrt{-1}} -4\pi \bold m \cdot \boldsymbol{\alpha} } d\boldsymbol{\alpha}
        \\
        &=\int_{\mathcal{D}^{\RR}} \tilde g_n(\boldsymbol{\alpha}) e^{\frac{n}{4\pi \sqrt{-1}}(f(\boldsymbol{\alpha})- 4\pi \bold m\cdot \boldsymbol{\alpha} ) } d\boldsymbol{\alpha},
    \end{align*}
    where 
    \begin{align*}
        \tilde g_n(\boldsymbol{\alpha}) = \frac{\left(\sum_{j=1}^{k_0}\frac{\partial^{2k_0} }{ \partial \alpha_j^{2k_0}} \right) \left( \psi(\boldsymbol{\alpha}) g_n(\boldsymbol{\alpha})e^{\frac{n}{4\pi \sqrt{-1}} f(\boldsymbol\alpha)} \right)}{e^{\frac{n}{4\pi \sqrt{-1}} f(\boldsymbol{\alpha}) }} 
    \end{align*}
    is a smooth function independent of $\bold m$ and has the form 
    \begin{align*}
        \tilde g_n(\boldsymbol{\alpha}) = \tilde g(\boldsymbol{\alpha}) n^{2k_0}  + O(n^{2k_0 - 1})
    \end{align*}
    for a smooth function $\tilde g(\boldsymbol{\alpha})$ independent of $n$. Therefore, on a compact subset of $\mathcal{D}^{\CC}, \tilde g_n(\boldsymbol{\alpha})/n^{2k_0}$ is bounded from above by some constant $C>0$ independent of $\bold m$. Then,
    \begin{align*}
        \sum_{\bold m \neq \bold K(J)} |\hat{G}_n(\boldsymbol{\alpha}, \bold m)| 
        &= \left(\frac{n}{2\pi} \right)^{k_0} \sum_{\bold m \neq \bold K(J)}
        \left|\int_{\mathcal{D}^{\RR}} \psi(\boldsymbol{\alpha}) g_n(\boldsymbol{\alpha}) e^{\frac{n}{4\pi \sqrt{-1}} (f(\boldsymbol{\alpha}) - 4\pi \bold m \cdot \boldsymbol{\alpha}) } d\boldsymbol{\alpha} \right|
        \\
        &= \left(\frac{n}{2\pi} \right)^{k_0} \sum_{\bold m \neq \bold K(J)}
        \frac{1}{m_1^{2k_0} + \dots + m_{k_0}^{2k_0}}
        \left|\int_{\mathcal{D}^{\RR}} \tilde g_n(\boldsymbol{\alpha}) 
        e^{\frac{n}{4\pi \sqrt{-1}} (f(\boldsymbol{\alpha}) - 4\pi \bold m \cdot \boldsymbol{\alpha}) } d\boldsymbol{\alpha} \right|
        \\
        &\leq \left(\frac{n}{2\pi} \right)^{k_0} 
        \sum_{\bold m \neq \bold K(J)}
        \frac{3AC}{\sum_{l=1}^{k_0} m_l^{2k_0}}
        O \left(e^{\frac{n}{4\pi }(\Vol (M_{\varphi}) - \epsilon )}  \right) \leq O\left( e^{\frac{n}{4\pi} (
         \Vol( M_{\varphi}) - \tfrac{\epsilon}{2})} \right)
    \end{align*}
    where $A$ is a constant depending on the area of the integral. 
    The summation converges because $\sum_{\bold m \neq \bold 0} \frac{1}{m_1^{2k_0} + \dots + m_{k_0}^{2k_0}}$ does.
\end{proof}
Therefore, the integral with leading Fourier coefficients can be split into multiple parts ($2^{k_0}$ parts to be precise) and for each of these deformed domains, we have an estimate. 
\begin{lemma}
    For a generic character $[\tilde r]$ in the same component as the hyperbolic character $[\tilde r_{\text{hyp}}]$, the sum
    \begin{align*}
        \sum_{J \in \{ 0,1\}^{k_0}} g_n(\boldsymbol{\alpha}^J) \neq 0.
    \end{align*}
\end{lemma}
\begin{proof}
    Recall the definition of the function $g_n$ \ref{g_n} at $\boldsymbol \alpha^J$
    \begin{align*}
        g_n(\boldsymbol{\alpha}^J) = e^{i\bold L(\boldsymbol{\alpha}^J)} \prod_{k=1}^{k_0} 
        \left(1-i^n e^{\frac{U_k}{2}}\right)^{j_k}
        \left(1-i^n e^{\frac{\hat{U}_k}{2}}\right)^{j_k} 
        \left( \frac{1 + e^{-2i\alpha_k}}{2}\right)^{1-\frac{U_k + \hat{U}_k }{4\pi i}} 
        e^{-\frac{\alpha_k (V_k + \tilde V_k)}{2\pi}} 
        \\
        = e^{i\bold L(\boldsymbol{\alpha}^{\bold 0} + \pi J)} 
        \prod_{k=1}^{k_0} \left( \frac{1 + e^{-2i\alpha_k}}{2}\right)^{1-\frac{U_k + \hat{U}_k }{4\pi i}} 
        e^{-\frac{\alpha_k (V_k + \tilde V_k)}{2\pi}} 
        \prod_{k=1}^{k_0} \left(1-i^n e^{\frac{U_k}{2}}\right)^{j_k}
        \left(1-i^n e^{\frac{\hat{U}_k}{2}}\right)^{j_k} .
    \end{align*}
    Therefore,
    \begin{align*}
        \sum_{J \in \{0,1\}^{k_0}} g_n(\boldsymbol{\alpha}^J) =& e^{i\bold L(\boldsymbol{\alpha}^{\bold 0})} 
        \prod_{k=1}^{k_0}  \left( \frac{1 + e^{-2i\alpha_k}}{2}\right)^{1-\frac{U_k + \hat{U}_k }{4\pi i}} 
        e^{-\frac{\alpha_k (V_k + \tilde V_k)}{2\pi}}
        \\
        &\times \sum_{J \in \{0,1\}^{k_0}} (-1)^{\bold L(J)} 
        \prod_{k=1}^{k_0} \left(1-i^n e^{\frac{U_k}{2}}\right)^{j_k}
        \left(1-i^n e^{\frac{\hat{U}_k}{2}}\right)^{j_k} 
        \\
        =& e^{i\bold L(\boldsymbol{\alpha}^{\bold 0})} 
        \prod_{k=1}^{k_0}  \left( \frac{1 + e^{-2i\alpha_k}}{2}\right)^{1-\frac{U_k + \hat{U}_k }{4\pi i}} 
        e^{-\frac{\alpha_k (V_k + \tilde V_k)}{2\pi}}
        \\
        &\times
        \left(1 + (-1)^{\hat{l}_1} \left(1-i^n e^{U_1/2}\right)\right) \left(1 + (-1)^{\hat{l}_2}\left(1-i^n e^{U_{k_0}/2}\right) \right)
        \prod_{k=2}^{k_0-1}\left(2- i^n e^{U_k/2}\right)
    \end{align*}
    Let's try to see the possible cases when the sum can be zero.  
    \begin{enumerate}
        \item $e^{-2i\alpha_k} = -1$ for any $k=1, \dots, k_0$. However, $e^{-2i\alpha_k} = a_k^{-1} = e^{U_k} \neq -1$. Therefore, this can not happen.
        \item $2 = i^n e^{U_k/2}$ for any $k=2, \dots, k_0-1$. This would imply $e^{U_k} = -4 = -z_k$ where $z_k$ is the shape parameter. For a geometric triangulation, the imaginary part of $z_k$ is positive. Thus, this can not happen either.
        \item If $\hat{l}_1$ is even or $\hat{l}_2$ is even, we get case 2. Using the same argument as before, we see that the term $(2-i^n e^{U_1/2}) \neq 0$ or $(2-i^n e^{U_{k_0}/2}) \neq 0$.
        \item If $\hat{l}_1$ or $\hat{l}_2$ is odd, then we are left with $-i^n e^{U_1/2}$ or $-i^n e^{U_{k_0}/2}$, which are non-zero.
    \end{enumerate}
    Thus, we see that the sum is never equal to 0, as long as we have a geometric triangulation.
\end{proof}
\begin{proof} [Proof of theorem \ref{asymptotics}]
    Thus our invariant $|\tr \Lambda^q_{\varphi, \tilde r}| $ can be written as
    \begin{align*}
        &=\left| \frac{1}{n_{k_0/2} \prod_{k=1}^{k_0} D^q(u_k)^{\frac{1}{n}} D^q(\hat{u}_k)^{\frac{1}{n}}} \sum_{\bold m \in \ZZ^{k_0}} \left(\frac{n}{2\pi}\right)^{k_0} \int_{\mathcal{D}^{\RR}} G_n(\boldsymbol{\alpha}) e^{\sqrt{-1}n \bold m\cdot \boldsymbol{\alpha}} d\boldsymbol{\alpha} 
        + O\left(e^{\frac{n}{4\pi} \Vol (M_{\varphi}) - \epsilon}\right) \right| 
        \\
        &= \frac{n^{k_0/2}}{(2\pi)^{k_0}\prod_{k=1}^{k_0} \left|D^q(u_k) D^q(\hat{u}_k) \right|^{\frac{1}{n}} }
        \left|
        \frac{(8\pi^2)^{k_0/2}}{ n^{k_0/2} \sqrt{(-1)^{k_0}\det \Hess (f) (\boldsymbol\alpha^{\bold 0}) }}
        \sum_{J \in \{0,1\}^{k_0}} g_n(\boldsymbol{\alpha}^J) e^{\frac{n}{4\pi} \Vol (M_{\varphi})} \right|
        \\
        & = \frac{1}{\prod_{k=1}^{k_0} \left|D^q(u_k) D^q(\hat{u}_k) \right|^{\frac{1}{n}} }
        \frac{\left| \sum_{J \in \{0,1\}^{k_0}} g_n(\boldsymbol{\alpha}^J) \right|}{\sqrt{(-1)^{k_0}\det \Hess (f) (\boldsymbol\alpha^{\bold 0}) }}
        e^{\frac{n}{4\pi} \Vol (M_{\varphi})} 
        \left( 1 + O\left( \frac{1}{n} \right) \right)
        \\
        & = c e^{\frac{n}{4\pi} \Vol (M_{\varphi})} \left( 1 + O\left( \frac{1}{n} \right) \right),
    \end{align*}
    where c $\neq 0$ is a constant defined by
    \begin{align*}
        c = \frac{ \left| \sum_{J \in \{0,1\}^{k_0}} g_n(\boldsymbol{\alpha}^J) \right| }{\sqrt{(-1)^{k_0}\det \Hess (f) (\boldsymbol\alpha^{\bold 0}) }\prod_{k=1}^{k_0} \left|D^q(u_k) D^q(\hat{u}_k) \right|^{\frac{1}{n}} }.
    \end{align*}
\end{proof}
\begin{corollary}
    When the diffeomorphism word is $LR$, we get the same potential function (with a factor of 2), and therefore the same critical point as \cite{BWY2}. Therefore, the volume is $4v_3$ and the traingulation of mapping torus has four tetrahedra where the volume of each of them is $v_3$. 
\end{corollary}
\begin{remark}
    The condition that $r$ is generic for uniqueness of representations from the Kauffman Bracket Skein algebra is also hoped to be extended to all irreducible characters (like in the case of closed surfaces).
\end{remark}

\end{document}